\documentclass{article}

\usepackage{amssymb,amsmath,amsthm,amsfonts}

\usepackage{graphicx,epstopdf,subfig}

\usepackage{fullpage}

\theoremstyle{plain}
\newtheorem{theorem}{Theorem}[section]

\newtheorem{proposition}[theorem]{Proposition}

\theoremstyle{definition}

\theoremstyle{remark}
\newtheorem{remark}{Remark}

\begin{document}

\title{Some extensions of the $\varphi$-divergence moment closures for the radiative transfer equation}

\author{M.R.A.~Abdelmalik\footnote{Department of Mechanical Engineering, Eindhoven University of Technology, Groene Loper 3, 5612 AE Eindhoven, Netherland}, 
Z.~Cai\footnote{Department of Mathematics, National University of Singapore, 10 Lower Kent Ridge Road, Singapore 119076.}~~and
T.~Pichard\footnote{CMAP, CNRS, École polytechnique, Institut polytechnique de Paris, 91120, Palaiseau, France}}

\date{}

\maketitle

\begin{abstract}
    The $\varphi$-divergence-based moment method was recently introduced~\cite{PhiRTE1} for the discretization of the radiative transfer equation. At the continuous level, this method is very close to the entropy-based $M_N$ methods and possesses its main properties, i.e. entropy dissipation, rotational invariance and energy conservation. However, the $\varphi$-divergence based moment systems
    are easier to resolve numerically due to the improved conditioning of the discrete equations. Moreover, exact quadrature rules can be used to compute moments of the distribution function, which enables the preservation of energy conservation, entropy dissipation and rotational invariants, discretely. In this paper we consider different variants of the $\varphi-$divergence closures that are based on different approximations of the exponential function and the Planck function. We compare the approximation properties of the proposed closures in the numerical benchmarks.
\end{abstract}

\textit{Keywords:} Radiative transfer equation, Method of moments, $\varphi$-divergence

\section{Introduction}
\label{sec:intro}
%context: RTE, numerics, properties to catch
This paper is a follow-up to~\cite{PhiRTE1} and is concerned about the discretization of the radiative transfer equation (RTE). This kinetic equation consists in a transport at a velocity $\Omega$ with a constant norm combined with a collision operator. Among the properties that this model satisfies, we want at least to preserve at the discrete level: the conservation of energy, the dissipation of a convex entropy and the rotational invariance.  

%State-of-the-art 
The numerical methods commonly used for this model include the Monte-Carlo solvers (DSMC;\cite{Carlson_book,Lewis_Miller_book,pomraning_book,mihalas_book}) and the discrete ordinates method (DOM;\cite{Carlson_book,Lewis_Miller_book,pomraning_book,mihalas_book}), but those are computationnally very expensive, not rotationally invariant and incapable to capture the equilibria distribution the system converges to. Our work falls within the context of the method of moments which is an efficient alternative. The most popular models in this family are the polynomial $P_N$ models (\cite{Chandrasekhar_book,pomraning_book,spectral_meth1,spectral_meth2}) and the entropy-based $M_N$ models (\cite{Minerbo,levermore}) which rely on approximating the $\Omega$-dependencies of the solution respectively by polynomials or by a distribution minimizing the entropy under moment constraints. Those are efficients, but the first fails at modelling beams, the second requires high computational costs due to need to solve large numbers of (potentially ill-conditionned) optimization problems. Those were an inspiration for many other techniques developed in this field, including the simplified models (\cite{Frank_SPN,McClaren_SPN}), the flux-limited diffusion (\cite{Olson_diffusion,Humbird_diffusion}), the interpolative methods (\cite{pichard_m2,Li_B2,groth_m2}) and others (see e.g.~\cite{Schneider_KN,PiN}). 

%recall brief construction of $\varphi$-div \\
Recently, the entropy-based method of moment was simplified by exploiting $\varphi$-divergence techniques (\cite{czizar}) while preserving its main properties, originally in the context of rarefied gases (\cite{abdelmalik_thesis,abdelmalik}) and for the RTE in~\cite{PhiRTE1}. This method consists in a Galerkin approximation where the test function space chosen to be polynomials set and  the approximation set is chosen to be a non-linear renormalization applied to the same polynomial set. This renormalization is chosen to be a high degree polynomial approximation of the exponential for the final model to dissipate an approximation of Boltzmann entropy. This model was shown to preserve the three aforementionned properties of conservation of energy, rotational invariance and entropy dissipation. Therefore, it possesses the same properties as $M_N$ models but the optimization problem to solve requires a lower cost and they rely on exact quadrature rules and therefore preserves further in the construction of the discretization the rotational invariance. 

%tackle two problems: faster convergence toward $\exp$ and convergence toward over $(\eta')^{-1}$
The present work exploits the versatility of the method in order to address two issues: First, when the degree of the renormalization mapping tends to infinity, our method falls back onto the $M_N$ method based on the Boltzmann entropy, it captures therefore exactly the beam distributions but requires to solve worse-conditioned optimization problems. We would like to adapt our method in order to choose the compromise between the acccuracy in the beam regime and the condition of the optimization problems to solve. Second, Boltzmann entrpopy is often considered in radiative transfer by analogy with the rarefied gases models, but other entropies are more physically relevant in this context. We would like to adapt our method such that it converges toward other entropy-based models. These two problems are tackled by constructing other polynomial renormalization mappings, and they eventually preserve the aforementionned properties. 

% organization
The paper is organized as follows. The next section recalls the radiative transfer equation and its properties. The next recalls the construction of our $\varphi$-divergence moment closure. The novel construction and study of polynomial renormalization mappings arises in Section~\ref{sec:LCBC}. Section~\ref{sec:numerics} is devoted to numerical experiments with the present method and Section~\ref{sec:concl} to conclusion.

\section{Radiative transfer equation}
\label{sec:RTE}
\begin{subequations}
We aim at solving the radiative transfer equation (RTE)
\begin{equation} 
  \partial_t I + \Omega \cdot \nabla_x I = LI := \sigma\left(\frac{1}{4\pi}\int_{\mathbb{S}^2} I \,\mathrm{d}\Omega - I\right),
  \label{eq:RTE}
\end{equation}
where the unknown $I(t,x,\Omega)$ is the radiative intensity depending on $\Omega\in \mathbb{S}^2$, $x\in \mathbb{R}^3$ and $t\in ]0,T[$. This equation is only supplemented with initial condition
\begin{equation} I(t=0,x,\Omega) = I^0(x,\Omega), \label{eq:RTE_IC}\end{equation}
and we still consider unbounded spatial domain in order to avoid the difficulties emerging with boundary conditions.
\label{sys:RTE}
\end{subequations}

This problem is well-posed and preserves the integrability of the initial condition~(\cite{Dautray-Lions}) in the sense: if $I^0 \in \mathcal{L}^p(\mathbb{R}^d\times\mathbb{S}^2)$, then there exists a unique function $I\in \mathcal{C}\left((0,T);\mathcal{L}^p(\mathbb{R}^d\times\mathbb{S}^2)\right)$ satisfying~\eqref{sys:RTE}. Furthermore, if $I^0 \ge 0$, then $I\ge 0$. In the following, we focus on $\mathcal{L}^1$ solutions.

\begin{subequations}
This equation is known to dissipate any entropy, i.e. for all convex scalar functions $\eta$, then 
\begin{gather} 
  \partial_t H(I) + \nabla_x \cdot G(I) = S(I) \le 0,  \qquad
  H(I) = \int_{\mathbb{S}^2} \eta(I(\Omega))d\Omega, \\
  G(I) = \int_{\mathbb{S}^2} \Omega \eta(I(\Omega))d\Omega, \qquad S(I) = \int_{\mathbb{S}^2} \eta'(I(\Omega)) LI(\Omega)d\Omega.  
\end{gather}
\label{eq:entropy}
\end{subequations}

In the present case, the space of collisional invariants 
\[ \mathtt{C} = \left\{ f \text{ s.t. } \int_{\mathbb{S}^2} f(\Omega)LI(\Omega)d\Omega = 0\right\} \] is one-dimensional and composed only of the isotropic functions 
\[ f\in \mathtt{C} \quad\Leftrightarrow\quad f(\Omega) = \frac{1}{4\pi} \int_{\mathbb{S}^2} f(\Omega)d\Omega. \]
Therefore, for all convex function $\eta$, $H$ is minimum when $I$ is isotropic and the system converges toward such equilibria. 

For the construction of entropy-based moment closure as below, the Boltzmann-Shannon (afterward denoted with the subscript $BS$) entropy
\[ \eta_{BS}(I) = I\log I \]
is often used as a comparison of the kinetic model~\eqref{eq:RTE} with rarefied gas models. But the Bose-Einstein (afterward denoted with the subscript $BE$) entropy 
\[ \eta_{BE}(I) = I\log I - (I+1)\log(I+1) \] 
is more meaningful when considering more physically realistic collision models in~\eqref{sys:RTE}. Typically, interaction of the radiations with matter is often modeled by adding scattering and emission terms in~\eqref{eq:RTE} (see e.g.~\cite{lowrie-morel,pomraning_book,mihalas_book}) which are equivalent, using Stefan's law, to a relaxation term toward a Planck function $(\eta_{BE}')^{-1}$ which parameters depend on matter temperature and radiations frequencies. Therefore, dissipating $\eta_{BE}$ is more relevant in this context than $\eta_{BS}$, and moment closures should be adapted to such other types of entropy. 

Finally, Equation~\eqref{sys:RTE} was shown to be rotationnally invariant, meaning that its solution $I$ satisfies
\[ (\partial_t I)(\mathcal O \Omega) = \partial_t ( I(\mathcal O \Omega)), \quad (\Omega \cdot \nabla_x I)(\mathcal O \Omega) = (\mathcal O \Omega) \cdot \nabla_x (I(\mathcal O \Omega)), \quad (LI)(\mathcal O \Omega) = L (I(\mathcal O \Omega)), \]
for all rotation matrices $\mathcal{O} \in SO(3)$.

\section{$\varphi$-divergence-based moment equations}
\label{sec:phidiv}
We recall here the construction of the moment closure from~\cite{PhiRTE1} and the problems tackled in the present work. 

\subsection{Construction of the Galerkin framework}
The moment system is obtained as a Galerkin approximation of~\eqref{eq:RTE} in the $\Omega$ variable. This formulation requires three elements, the choices and properties are recalled here:
\begin{itemize}
\item A \textit{finite dimensional} test functions space $M$, which must be a \textit{subset of the solution set dual} $\mathcal{L}^\infty(\mathbb{S}^2)$. In order to preserve those properties at the numerical level, $M$ must \textit{contain the collision invariants} $1 \in M$; and $M$ must be \textit{rotational invariant}. The natural choice to satisfy both properties is the set of polynomials up to a certain degree $N$
\[ M := \mathbb{P}_N(\mathbb{S}^2).\]    
\item A renormalization map $\beta$ to account for non-linearity in the approximation. For a convex function $\eta$, choosing $\beta = (\eta')^{-1}$ corresponds to dissipating $\eta$ at the underlying kinetic level (see e.g.~\cite{levermore}). Especially, $\beta$ must be \textit{monotonically increasing} to match such an entropy. Natural choices include 
\begin{equation} \beta_{BS}(g) = \exp(g) = (\eta_{BS}')^{-1}(g), \qquad \beta_{BE}(g) = \frac{1}{\exp(g)-1} = (\eta_{BE}')^{-1}(g). \label{eq:beta_BS_BE}\end{equation}
In our work, we aim at imposing $\beta \in \mathbb{P}_K(\mathbb{R})$ with $K\ge 1$ for numerical quadrature (up to a sufficient order) to be exact, this provided a rotationally invariant algorithm in \cite{PhiRTE1}. We chose renormalization of the form 
\begin{equation} \beta_K(g) = \left(1+ \frac{g}{K}\right)^K = (\eta_K')^{-1}(g) \quad\text{with}\quad \eta_K(I) = K I \left(\frac{K}{K+1}I^{1/K}-1\right),\label{eq:def_betaK}\end{equation}
which are polynomial approximations of $\exp$, and monotonically increasing for odd $K\ge 1$. We exhibit in the next section other choices. 
\item A finite dimensional trial functions space $V$ such that $\beta(V) \subset \mathcal{L}^1$ is a subset of the solution set. Again, $\beta(V)$ must \textit{contain the equilibrium distributions} $\mathtt{C} \subset \beta(V)$; and $V$ must be \textit{rotational invariant}, and we choose again
\[ V := \mathbb{P}_N(\mathbb{S}^2).\]  
\end{itemize}
Eventually, this yields seeking $g\in \mathbb{P}_N(\mathbb{S}^2)$ such that for all $m\in\mathbb{P}_N(\mathbb{S}^2)$
\begin{subequations}
\begin{gather}
    \forall (t,x)\in(0,T)\times\mathbb{R}^d,\quad  \int_{\mathbb{S}^2} m(\Omega) \left[ (\partial_t + \Omega \cdot \nabla_x) \beta(g) - L\beta(g(t,x,\cdot))(\Omega)\right]d\Omega = 0, \label{eq:moment_system} \\
    \forall x\in\mathbb{R}^d,\quad \int_{\mathbb{S}^2} m(\Omega) \left[\beta(g(t=0,x,\Omega)) - I^0(x,\Omega)\right]d\Omega = 0. \label{eq:moment_IC} 
\end{gather}
\label{eq:mom_gal}
\end{subequations}
Let $\boldsymbol{m}$ denote a vector of all of the basis functions $m(\Omega)\in \mathbb P_N(\mathbb S^2)$, under the constraints mentioned in the last paragraph. Then (\ref{eq:mom_gal}) can be expressed as a system of moment equations
\begin{subequations}
\begin{align}
    \partial_t \mathbf{U} + \operatorname{div}_x(\mathbf{F}(\mathbf{U})) &= \mathbf{L}\mathbf{U}, 
    \label{eq:mom_sym_sys}
    \\ 
    (\mathbf{U}, \mathbf{F}(\mathbf{U}),\mathbf{L}\mathbf{U}) &= \int_{\mathbb{S}^2} \mathbf{m}(\Omega)\left(\beta(\boldsymbol{\lambda}^T\mathbf{m}(\Omega)),\ \Omega^T\beta(\boldsymbol{\lambda}^T\mathbf{m}(\Omega),\ L\beta(\boldsymbol{\lambda}^T\mathbf{m}(\Omega)\right) \mathrm{d}\Omega,
    \label{eq:mom_cons}
\end{align}
\end{subequations}
which possesses a symmetrizer constructed from $\eta$ (the anti-derivative of $\beta^{-1}$) and $\boldsymbol{\lambda}$ are the so-called entropic variables (\cite{Godlewski_Raviart_book}) in which (\ref{eq:mom_sym_sys}) can be written in the so-called symmetric hyperbolic form. 
Therefore, (\ref{eq:mom_sym_sys}) possesses a convenient structure for the study of its well-posedness (\cite{Kawashima-Yong}). 
\begin{remark}
The entropic variables $\boldsymbol{\lambda}$ introduced in (\ref{eq:mom_cons}) can also be conceived of as Lagrange multipliers that enforce the moment constraints in the so-called entropy minimization problem (\cite{levermore})
\begin{equation} \label{eq:entmin}
    \text{Find } \mathrm{argmin}\left\{ H(h): \int_{\mathbb S^2} \boldsymbol{m} I d\Omega =\int_{\mathbb S^2} \boldsymbol{m} h d\Omega \right\}.
\end{equation}
\end{remark}

\subsection{Position of the problem}
In this paper, we investigate several modifications of the renormalization map~\eqref{eq:def_betaK} and its impact on the closure~\cite{PhiRTE1}. The objective is to alter the convergence of the $\varphi$-divergence solution when $K \rightarrow +\infty$ either to accelerate the convergence or to modify the limit value:

First, when considering the moments of a Dirac distribution, e.g. $\mathbf{U} = \mathbf{m}(e_1) = \int_{\mathbb{S}^2} \mathbf{m} \delta_{e_1}$, numerical experiments in~\cite{PhiRTE1} showed that the reconstruction $\beta_K$ does converge in $H^{-2}$ toward the distribution $\delta_{e_1}$ when $K \rightarrow +\infty$. However, the rate of convergence is slow. This slow rate of convergence is also illustrated on Fig.~\ref{fig:renorm_beta_K} where the function $\beta_K(x)$ is plotted for various $K$ odd together with its limits $\lim\limits_{K\rightarrow \infty} \beta_K = \exp$. 
\begin{figure}[h!]
    \centering
    \includegraphics[width=.7\textwidth]{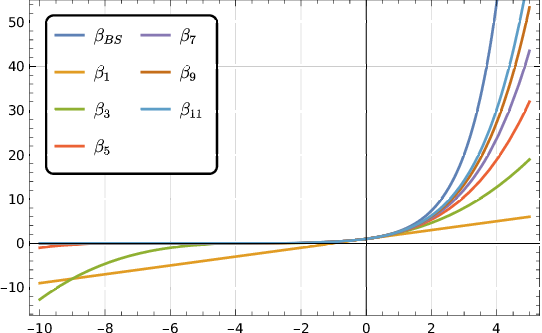}
    \caption{Renormalization mappings $\beta_K$ for odd $K$ and exponential function.}
    \label{fig:renorm_beta_K}
\end{figure}

 Formally, the Dirac distribution corresponds to the values $x \rightarrow -\infty$ and $x \rightarrow \infty$. Indeed, considering a Gaussian mollifier $\frac{1}{\sqrt{\pi\sigma}}\exp(-\frac{y^2}{\sigma})$, then $x=\frac{-y^2}{\sigma}$ takes for values $\pm\infty$ in the limit. 
 
 One observes that the sequence $(\beta_K)_{K\in\mathbb{N}}$ indeed converges pointwisely for bounded $x$ toward the exponential function, but this convergence is very slow. Furthermore, the rate of the exponential in the limit $x\rightarrow+\infty$ is not accurately captured and one needs high order $K$, and therefore higher complexity, to capture such large values. Similarly, the zero limit when $x\rightarrow -\infty$ can not be reached by any $\beta_K$ function with finite $K$ as they are polynomials and cannot have bounded value in $-\infty$. Eventually, such Dirac distributions can only be approximated and we only aim at improving the range of accuracy of such approximations.  

Second, when $K\rightarrow +\infty$, the sequence $(\beta_K)_{K\in\mathbb{N}}$ of approximations may only converge toward the exponential. As mentioned in the previous section, other types of equilibrium can be expected from the solution of the RTE, typically the Planck distribution $\beta_{BE}$ defined in~\eqref{eq:beta_BS_BE}. The $\beta_K$ approximation does not possess the flexibility to converge toward other $\beta = (\eta')^{-1}$ functions. 

Therefore, the objective in the next section is to provide another type of approximations which is flexible enough to control the convergence when $K\rightarrow \infty$, i.e. both the convergence rate and the limit function. 

\section{Other monotonically increasing polynomial approximations}
\label{sec:LCBC}
The solutions considered in this paper consist in \textit{polynomial} approximations. The integral of such polynomial functions can be computed exactly using an appropriate quadrature rule. Therefore, the coefficients $\boldsymbol{\lambda}$ of the approximation in~\eqref{eq:mom_sym_sys} can be obtained using Newton iterations~\cite{PhiRTE1} that can be computed exactly. Remark that rotation invariance is lost when constructing entropy-based closures (see e.g.~\cite{Hauck}) due to non-exact integration while the present choice of polynomial approximation allows to compute the integrals exactly. 

For the model to possess a convex entropy dissipated, we still need this approximation to be \textit{monotonically increasing}. Therefore, we require 
\[\beta' > 0.\] 

\subsection{Taylor expansion}
A first idea originated in the observation that the Taylor expansion of the exponential around zero converges faster (empirically) than the sequence $(\beta_K)_{K\in\mathbb{N}}$. This consists in writing in the shifted monomial basis $\mathbf{b} = \mathbf{b}^K(x) := \left(1,\dots,(x-x_0)^K\right)^T$
\begin{equation} 
    T_K(x) = \sum\limits_{k=0}^K \alpha_k (x-x_0)^k = \boldsymbol{\alpha}^T \mathbf{b}
\end{equation}
where $\alpha_k = \frac{\beta^{(k)}(x_0)}{k!}$ for a generic function $\beta$. We provide a simple characterization of monotonically increasing polynomials of this form. 
\begin{proposition}
    Suppose that $\beta \in \mathcal{C}^{2K+2}$ is such that $\beta^{(i)} \ge 0$ for all $1\le i\le 2K+2$. Then the polynomial $T_{2K+1}$ is monotonically increasing.   
\end{proposition}
\begin{proof}
    This simply follows from the Taylor formula with remainder of the derivative of $f$: there exists $\xi \in [x,x_0]$ such that
    \[ \beta'(x) - T_{2K+1}'(x) = \frac{\beta^{(2K+2)}(\xi)}{(2K+2)!}(x-x_0)^{2K+1}, \]
    which is negative for $x < x_0$. Therefore, for $x \le x_0$
    \[ \sum\limits_{i=0}^{2n} \frac{f^{(i+1)}(x_0)}{i!} (x-x_0)^i  \ge f'(x) \ge 0. \]
    The derivative $T_{2K+1}'(x)$ is also positive for $x>x_0$, then it is monotonically increasing. 
\end{proof}
All the derivatives of the exponential $\beta_{BS} = \exp$ are positive, then it satisfies this property and the polynomial 
\[ T_{2K+1}(x) = \sum\limits_{k=0}^{2K+1} \frac{e^{x_0}}{k!} (x-x_0)^k \]
is monotonically increasing for odd degree $2K+1$. The convergence radius of this sequence is infinite. Therefore, we have convergence of the approximation toward the desired results for all $x\in\mathbb{R}$
\[ T_{2K+1}(x) \underset{K\longrightarrow \infty}{\longrightarrow} \exp(x). \]

Concerning the Planck function $\beta_{BE}$ minimizing the Bose-Einstein entropy, it reads for $x\in\mathbb{R}^{*,-}$ 
\[ \beta_{BE}(x) = \frac{1}{e^{-x}-1} > 0,\]
and one verifies that its derivative satisfies $\beta_{BE}' = (1+\beta_{BE})\beta_{BE}$ such that all the successive derivatives of $\beta_{BE}$ are polynomials in $\beta_{BE}$ with positive coefficients. Especially, those derivatives are all strictly positive for all $x\in\mathbb{R}^{*,-}$ since $\beta_{BE}>0$. Therefore, the polynomial $T_{2K+1}$ of odd degree with the Planck function are also monotonically increasing. The convergence radius $\delta$ of this sequence remains bounded and it depends on the chosen point of expansion $x_0$. This radius $\delta < |x_0|$ since the function $\beta_{BE}$ is singular in zero (and undefined after). Therefore, we only have convergence of the approximation toward the desired results for all $x\in (x_0-\delta,x_0+\delta) \subsetneq \mathbb{R}^{*,-}$
\[ T_{2K+1}(x) \underset{K\to \infty}{\to} \beta_{BE}(x). \]
Especially, we do not have convergence $T_{2K+1}(x) \underset{K\to \infty}{\not\to} \beta_{BE}(x)$ for all points $x < 2x_0$.
These are illustrated numerically in Section~\ref{subsec:comparisons_approx} below, which even exhibit divergence for such values of $x$.

\subsection{Optimized parameters} 
Another idea is to minimize the $\mathcal{L}^2$ difference between the function $\beta$ to approximate ($\beta_{BS}$ or $\beta_{BE}$) and the polynomials of a given degree, with the constraint that the polynomials must be monotonically increasing. This can be mathematically represented as
\begin{align}
O_{2K+1} = & \ \underset{p \in \mathbb{P}_{2K+1}}{\text{argmin}} \ \frac{1}{2} \int_a^b |p(x) - \beta(x)|^2 \,\mathrm{d}x, \label{eq:def_L2_Opt}\\
& \ \text{subject to } p'(x) \geq 0, \quad \forall x \in \mathbb{R}. \nonumber
\end{align}

\subsubsection{Reformulation of the approximation}
In order to enforce the constraint, we use the fact that all non-negative one-variable polynomials can be represented as the sum of two squares, one of which has a lower degree than the other (\cite{Lasserre_book,Schmuedgen_book,Szego_book}). Thus, the derivative of the polynomial $O_{2K+1}(x)$ has the form:
\begin{displaymath}
O_{2K+1}'(x) = (a_0 + a_1 x + \cdots + a_K x^K)^2 + (b_0 + b_1 x + \cdots + b_{K-1} x^{K-1})^2.
\end{displaymath}
Integrating this equation gives
\begin{equation}
\label{eq:admissible_solution}
\begin{split}
O_{2K+1}(x) = C &+ \sum_{i=0}^K \sum_{j=0}^K (a_i a_j + b_i b_j) \frac{x^{i+j+1}}{i+j+1} \\
= C &+ \sum_{n=1}^{K+1} \sum_{i=0}^{n-1} (a_i a_{n-1-i} + b_i b_{n-1-i}) \frac{x^n}{n} \\ 
&+ \sum_{n=K+2}^{2K+1} \sum_{i=n-1-K}^{K} (a_i a_{n-1-i} + b_i b_{n-1-i}) \frac{x^n}{n},
\end{split}
\end{equation}
where we assumed $b_K = 0$. With this form of $O_{2K+1}$, we can turn the optimization problem into an unconstrained optimization problem. For simplicity, we define intermediate parameters
\begin{displaymath}
\alpha_0 = C, \qquad
\alpha_n = \frac{1}{n} \sum_{i=\max(0,n-1-k)}^{\min(k,n-1)} (a_i a_{n-1-i} + b_i b_{n-1-i}), \quad n=1,\cdots,2K+1,
\end{displaymath}
so that $\alpha_n$ is the $n$-th coefficient of the polynomial $O_{2K+1}$. Let
\begin{displaymath}
\boldsymbol{w} = (C, a_0, \cdots, a_K, b_0, \cdots, b_{K-1})^T, \qquad \boldsymbol{\alpha}(\boldsymbol{w}) = (\alpha_0, \alpha_1, \cdots, \alpha_{2K+1})^T.
\end{displaymath}
Then the Jacobian matrix $J = \partial \boldsymbol{\alpha} / \partial \boldsymbol{w}$ has the following form:
\begin{displaymath}
J = \begin{pmatrix}
1 & 0 & 0 \\
0 & 2A & 2B
\end{pmatrix},
\end{displaymath}
where
\begin{displaymath}
A = \begin{pmatrix}
a_0               &  0                  & \dots               &        & \dots  & 0\\[5pt]
\frac{a_1}{2}     & \frac{a_0}{2}       & 0                   &        &        & \vdots \\[5pt]
\frac{a_2}{3}     & \frac{a_1}{3}       & \frac{a_0}{3}       & \ddots &        & \\
\vdots            & \frac{a_2}{4}       & \frac{a_1}{4}       & \ddots & \ddots & \vdots \\
\vdots            & \vdots              & \frac{a_2}{5}       & \ddots & \ddots & 0 \\
\frac{a_{K}}{K+1} & \vdots              & \vdots              & \ddots & \ddots & \frac{a_0}{K+1} \\
0                 & \frac{a_{K}}{K+2} & \vdots              &        & \ddots & \frac{a_1}{K+2} \\
\vdots            & \ddots              & \frac{a_{K}}{K+3} &        &        & \frac{a_2}{K+3} \\
                  &                     & \ddots              & \ddots &        & \vdots \\
                  &                     &                     &        & \ddots & \vdots \\
                  &                     &                     &        & \ddots & \frac{a_{K}}{2K+1}\\
                  &                     &                     &        &        & 0 \\
\vdots &  &  &  &  & \vdots \\
0 & \dots &  &  & \cdots & 0
\end{pmatrix},  \qquad
B = \begin{pmatrix}
b_0               &  0                  & \dots               &        & \dots  & 0\\[5pt]
\frac{b_1}{2}     & \frac{b_0}{2}       & 0                   &        &        & \vdots \\[5pt]
\frac{b_2}{3}     & \frac{b_1}{3}       & \frac{b_0}{3}       & \ddots &        & \\
\vdots            & \frac{b_2}{4}       & \frac{b_1}{4}       & \ddots & \ddots & \vdots \\
\vdots            & \vdots              & \frac{b_2}{5}       & \ddots & \ddots & 0 \\
\frac{b_{K-1}}{K} & \vdots              & \vdots              & \ddots & \ddots & \frac{b_0}{K} \\
0                 & \frac{b_{K-1}}{K+1} & \vdots              &        & \ddots & \frac{b_1}{K+1} \\
\vdots            & \ddots              & \frac{b_{K-1}}{K+2} &        &        & \frac{b_2}{K+2} \\
                  &                     & \ddots              & \ddots &        & \vdots \\
                  &                     &                     &        & \ddots & \vdots \\
                  &                     &                     &        & \ddots & \frac{b_{K-1}}{2K-1}\\
                  &                     &                     &        &        & 0 \\
\vdots &  &  &  &  & \vdots \\
0 & \dots &  &  & \cdots & 0
\end{pmatrix}.
\end{displaymath}

To solve the optimization problem, we first reformulate the objective function $\boldsymbol{w} \mapsto f(\boldsymbol{\alpha}(\boldsymbol{w}))$ as a function of the intermediate parameters $\boldsymbol{\alpha}$: 
\begin{displaymath}
\begin{split}
f(\boldsymbol{\alpha}) & = \frac{1}{2} \int_a^b \left| \sum_{n=0}^{2K+1} \alpha_n x^n  - \beta(x) \right|^2 \,\mathrm{d}x \\
&= \frac{1}{2} \boldsymbol{\alpha}^T M \boldsymbol{\alpha} - \boldsymbol{\beta}^T \boldsymbol{\alpha} + \frac{1}{2} \int_a^b \beta(x)^2 \,\mathrm{d}x.
\end{split}
\end{displaymath}
Note that the value of the last integral does not matter in our optimization problem, and the matrix $M$ and the vector $\boldsymbol{\beta}$ are given by
\begin{gather*}
 M_{i,j} = \frac{b^{i+j+1} - a^{i+j+1}}{i+j+1}, \qquad
\beta_j = \int_a^b x^j \beta(x) dx.
\end{gather*}
In the case of the exponential $\beta = \beta_{BS}$, the second coefficient rewrites
\[\beta_j = (-1)^j \left[\Gamma(j+1, -b) - \Gamma(j+1,-a)\right], \]
where the incomplete $\Gamma$ function is well-implemented in standard numerical libraries. In the case $\beta=\beta_{BE}$, the integral can be represented using the polylogarithmic function $\operatorname{Li}_s(z)$:
\begin{displaymath}
  \beta_j = \sum_{k=0}^j \frac{(-1)^{j+k} j!}{k!} \left[
    b^k \operatorname{Li}_{j+1-k} (\mathrm{e}^b) - 
    a^k \operatorname{Li}_{j+1-k} (\mathrm{e}^a)
  \right].
\end{displaymath}

\subsubsection{Details on the numerical computations}
The minimum of $f(\boldsymbol{\alpha}(\boldsymbol{w}))$ is attained where the gradient anihilates. Then we need to solve the nonlinear system $\nabla_{\boldsymbol{w}} f(\boldsymbol{\alpha}(\boldsymbol{w})) = 0$, or
\begin{equation}
[J(\boldsymbol{w})]^T M \boldsymbol{\alpha}(\boldsymbol{w}) = 0. \label{eq:gradient_zero_OK}
\end{equation}
Since $J(\boldsymbol{w})$ is linear and $\boldsymbol{\alpha}(\boldsymbol{w})$ is quadratic, these are actually cubic equations in $\boldsymbol{w}$. 

In the following, we compute numerically some parameters $\boldsymbol{w}$ and $\boldsymbol{\alpha}(\boldsymbol{w})$ by using Newton's method for~\eqref{eq:gradient_zero_OK}. Note that the $f(\boldsymbol{\alpha}(\boldsymbol{w}))$ is a quartic function which is infinity in infinity, but it might not be convex at all $\boldsymbol{w} \in \mathbb{R}^{2K+2}$, and it may possess several local extrema. Therefore, we cannot guarantee the convergence of Newton's method toward a global minimum. In practice, we choose $500$ initial values of $\boldsymbol{w}$ randomly and pick the final solution with the minimum value of the objective function.

Using the density of the (positive) polynomials in $\mathcal{L}^2$, we obtain that this approximation converges in $\mathcal{L}^2(a,b)$ toward the desired $\beta$ function. However, the convergence is again restricted to the chosen interval $(a,b)$ and we can not guarantee convergence out of it. These are illustrated numerically in the next paragraph.    

\subsection{Comparing the monotonic polynomial approximations} 
\label{subsec:comparisons_approx}

\subsubsection{Approximation for the Boltzmann-Shannon entropy}
The approximations of the exponential $\beta_{BS}$ are compared in Fig.~\ref{fig:optimized_bs} for different degrees $2K+1$. The point of expansion is chosen to be $x_0 = 0$ for the Taylor approximation and the range for the optimized approximation is $[-L,L]$ with $L=1,3,5$. In general, by comparing the $\beta_{2K+1}$ model and the $T_{2K+1}$ model, the $\beta_{2K+1}$ model gives a better approximation for $x \in \mathbb{R}^{*,-}$, while the $T_{2K+1}$ model does better for $x \in \mathbb{R}^{*,+}$. For the optimized approximation, the function depends on the choice of the range $[-L,L]$. As expected, when $K$ is fixed, the function is approximated within the range $[-L,L]$ for smaller values of $L$. It balances the quality of the approximation for both the negative and positive parts, and generally looks better than both $\beta_{2K+1}$ and $T_{2K+1}$ in the plots with linear scales (left panel of Fig. \ref{fig:optimized_bs}). Similar phenomena can be observed in Fig.~\ref{fig:comparison_bs}, where all the results for the optimized approximation are given for $L = 5$. The log plots show a clear difference between the $O_{2K+1}$ model and the $\beta_{2K+1}$ and $T_{2K+1}$ models: for the latter, the entire approximation is below the exact exponential function, whereas the $O_{2K+1}$ approximation oscillates around the exponential, which is a typical behavior in spectral approximations.

Another interesting phenomenon that can be observed from Fig.~\ref{fig:optimized_bs} is that the $O_{2K+1}$ approximation seems to have a good capability in extrapolations. For instance, the red curve in Fig.~\ref{fig:K5_bs} is computed by minimizing the $\mathcal{L}^2$ distance~\eqref{eq:def_L2_Opt} only on the interval $[-1,1]$, but this approximation is also quite accurate for $x\in [-3,5]$. Since the choice of the approximation range usually depends on some preliminary estimation of the problem, this property allows us to use the $O_{2K+1}$ approximations without too much worries about the function values out of the chosen range in actual simulations. Here we conjecture that for any fixed interval, the $O_{2K+1}$ model will converge to the function $\beta_{BS}$ as $K$ tends to infinity.
\begin{figure}[h!]
\centering
\subfloat[$K=1$]{\includegraphics[width=.45\textwidth]{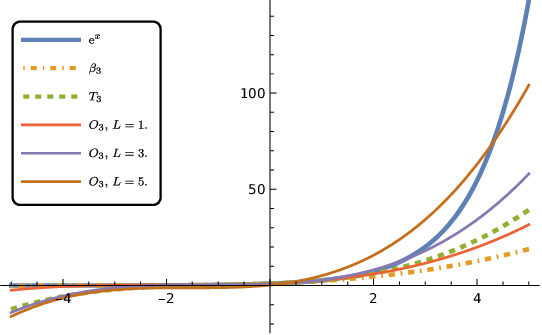}} \qquad 
\subfloat[$K=1$, log plot]{\includegraphics[width=.45\textwidth]{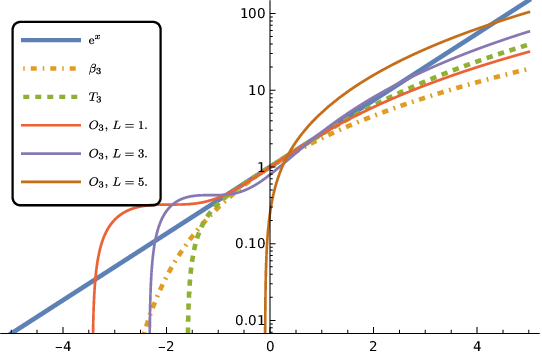}} \\
\subfloat[$K=3$]{\includegraphics[width=.45\textwidth]{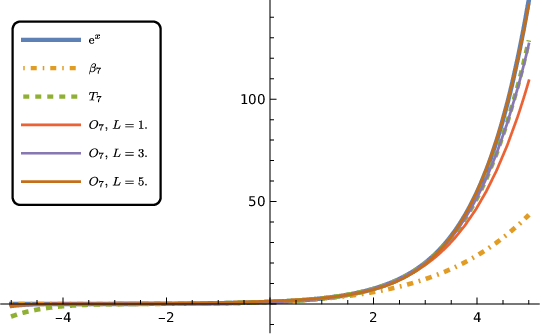}} \qquad
\subfloat[$K=3$, log plot]{\includegraphics[width=.45\textwidth]{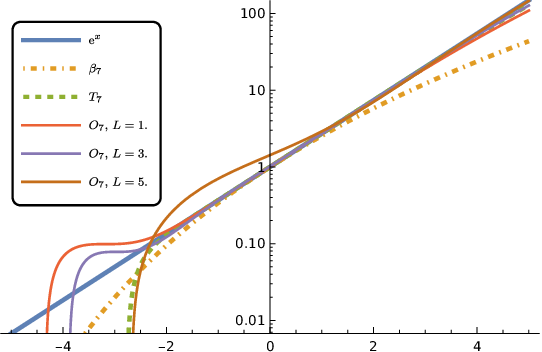}} \\
\subfloat[$K=5$]{\includegraphics[width=.45\textwidth]{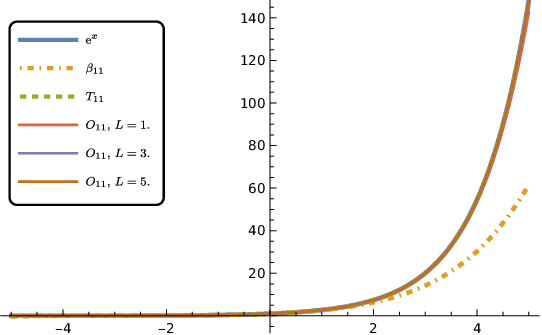}} \qquad
\subfloat[$K=5$, log plot]{\label{fig:K5_bs}\includegraphics[width=.45\textwidth]{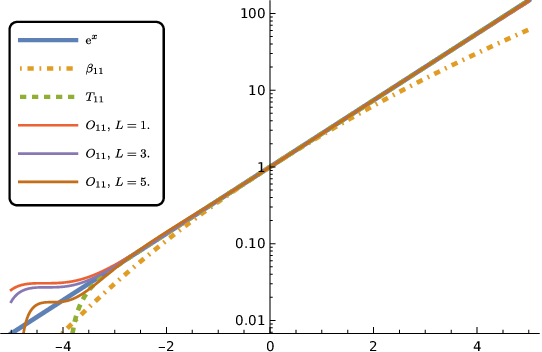}} 
\caption{$\beta_{2K+1}$, Taylor $T_{2K+1}$ and optimized $O_{2K+1}$ approximations of the exponential function $\beta_{BS}$.}
\label{fig:optimized_bs}
\end{figure}

\begin{figure}[h!]
\centering
\subfloat[$\beta_N$]{\includegraphics[width=.45\textwidth]{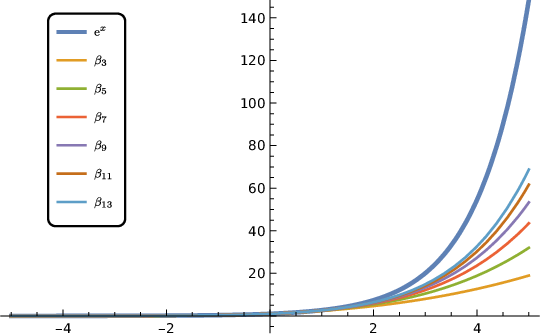}} \qquad
\subfloat[$\beta_N$, log plot]{\includegraphics[width=.45\textwidth]{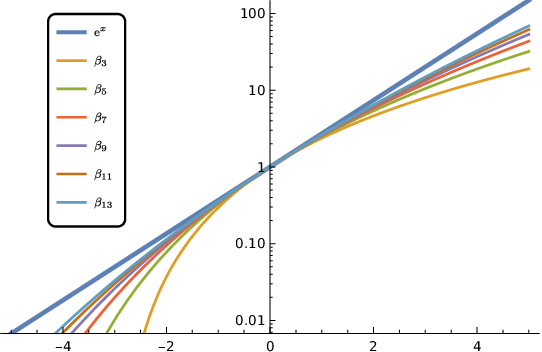}} \\
\subfloat[$T_N$]{\includegraphics[width=.45\textwidth]{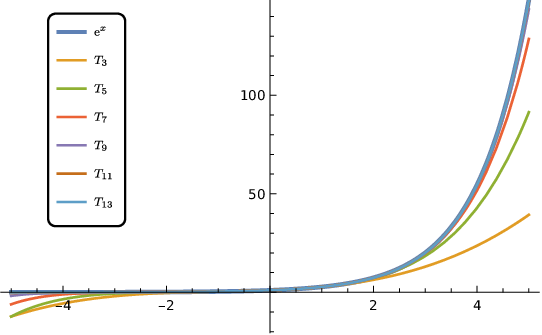}} \qquad
\subfloat[$T_N$, log plot]{\includegraphics[width=.45\textwidth]{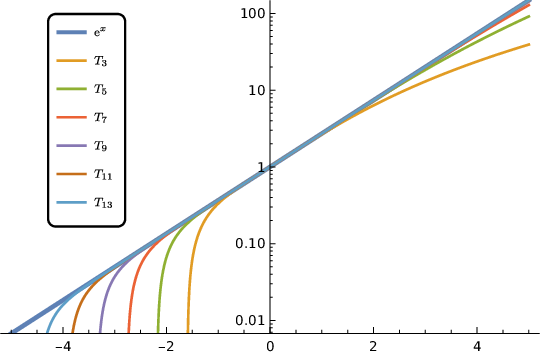}} \\
\subfloat[$O_N$]{\includegraphics[width=.45\textwidth]{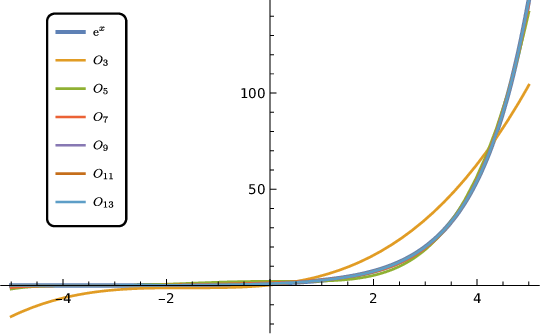}} \qquad 
\subfloat[$O_N$, log plot]{\includegraphics[width=.45\textwidth]{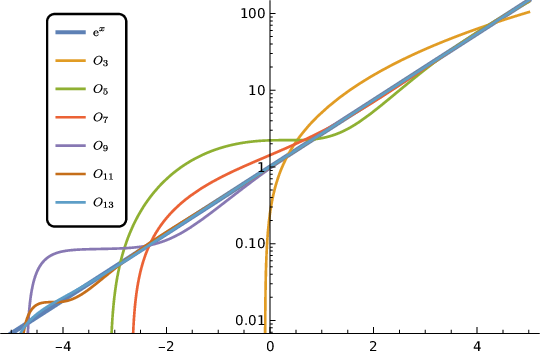}}
\caption{$\beta_{2K+1}$, Taylor $T_{2K+1}$ and optimized $O_{2K+1}$ approximations of the exponential function $\beta_{BS}$.}
\label{fig:comparison_bs}
\end{figure}

Fig.~\ref{fig:error_bs} plots the $\mathcal{L}^2(-L,L)$ difference between the exponential function $\beta_{BS}$ and the $O_{2K+1}$ approximations. Unsurprisingly, the error increases as $L$ increases and decreases as $K$ increases. For a fix $L$, the gaps between lines are nearly the same, showing the spectral convergence rate with respect to $K$. When $L$ gets larger, the gap becomes narrower, indicating slower convergence. For a fixed $K$, the figure implies that the error increases in the form of $L^{\alpha}$ for a certain value of $\alpha$ depending on $K$.
\begin{figure}[!ht]
\centering
\includegraphics[width=.8\textwidth]{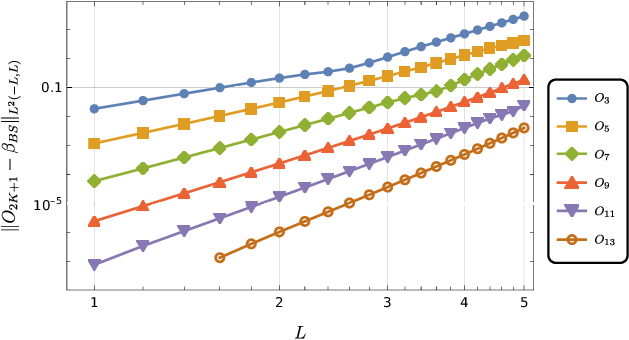}
\caption{Error plot for the $O_{2K+1}$ approximations of the exponential function $\beta_{BS}$}
\label{fig:error_bs}
\end{figure}

\subsubsection{Approximation for the Bose-Einstein entropy}
Similar experiments are done for the Planckian $\beta_{BE}$. The results are plotted in Fig.~\ref{fig:optimized_be} and \ref{fig:comparison_be}. Since the function is defined only for negative values, we choose the range of approximation to be $[-L, -1/L]$ with $L = 2,6,10$. In Fig.~\ref{fig:optimized_be}, three different choices of $x_0$ are considered for the Taylor expansion, and the choices are made with $x_0 = -(L + 1/L)/2$, which is the center of the interval $[-L, -1/L]$ used in the $O_{2K+1}$ approximation. The general behavior is similar to the case of Boltzmann-Shannon entropy: the optimized approximation fits the Planckian better within the range $[-L,-1/L]$, but it may perform worse than the Taylor approximation out of this interval. The convergence with respect to $K$ can be better observed in Fig.~\ref{fig:comparison_be}, where all Taylor series are expanded about the same point $x_0 = -3.08333$, and the value of $L$ is fixed to be $6$ for all optimized approximations. Compared with Taylor approximations, the optimized approach better approximates the part with larger function values, which suppresses the $\mathcal{L}^2$ error more efficiently. Note that both the $T_{2K+1}$ and $O_{2K+1}$ approximations are increasing functions across the entire real axis $\mathbb{R}$, despite the seemingly oscillatory behavior of $O_{2K+1}$ functions.

\begin{figure}[h!]
\centering
\subfloat[$K=1$]{\includegraphics[width=.45\textwidth]{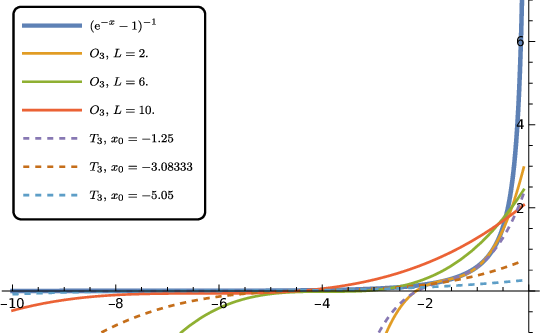}} \qquad 
\subfloat[$K=1$, log plot]{\includegraphics[width=.45\textwidth]{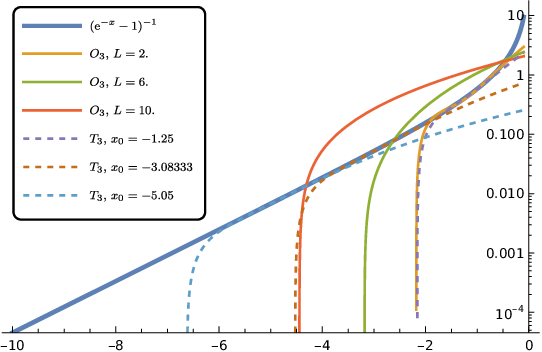}} \\
\subfloat[$K=3$]{\includegraphics[width=.45\textwidth]{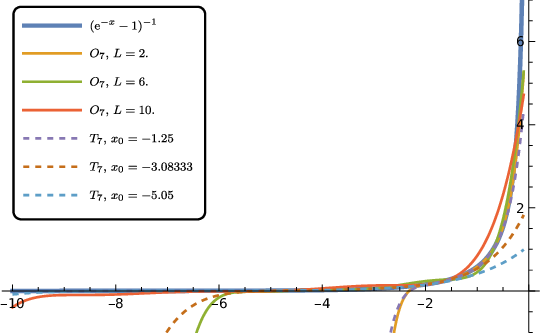}} \qquad
\subfloat[$K=3$, log plot]{\includegraphics[width=.45\textwidth]{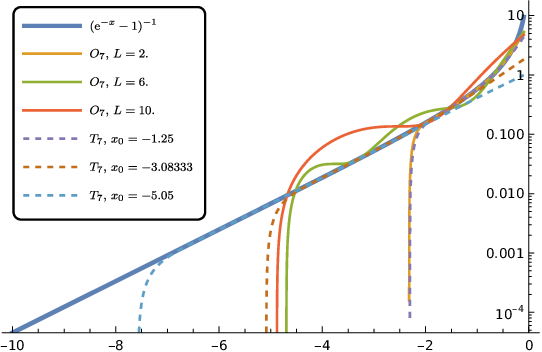}} \\
\subfloat[$K=5$]{\includegraphics[width=.45\textwidth]{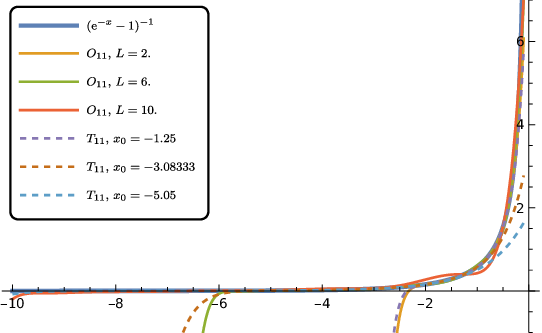}} \qquad
\subfloat[$K=5$, log plot]{\label{fig:K5_be}\includegraphics[width=.45\textwidth]{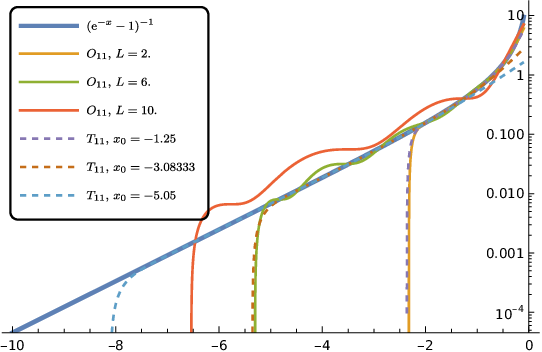}} 
\caption{Taylor $T_{2K+1}$ and optimized $O_{2K+1}$ approximations of the Planckian $\beta_{BE}$.}
\label{fig:optimized_be}
\end{figure}

\begin{figure}[h!]
\centering
\subfloat[$T_{2K+1}$]{\includegraphics[width=.45\textwidth]{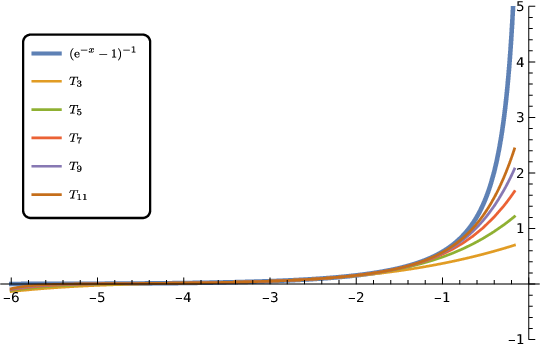}} \qquad
\subfloat[$T_{2K+1}$, log plot]{\includegraphics[width=.45\textwidth]{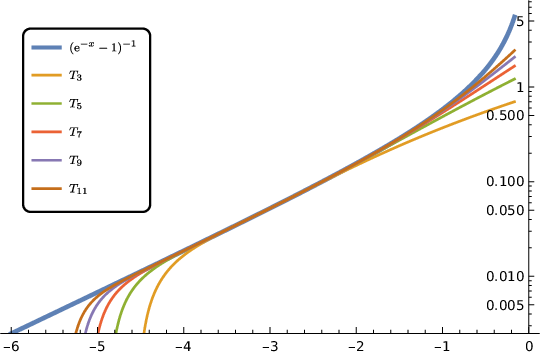}} \\
\subfloat[$O_{2K+1}$]{\includegraphics[width=.45\textwidth]{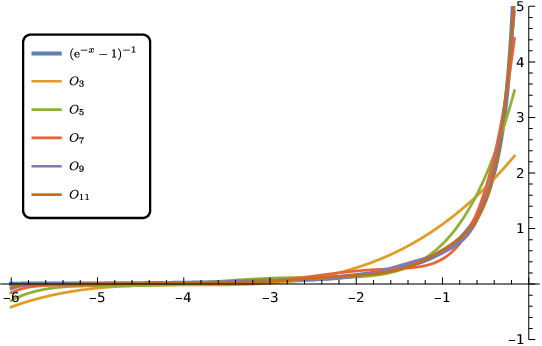}} \qquad 
\subfloat[$O_{2K+1}$, log plot]{\includegraphics[width=.45\textwidth]{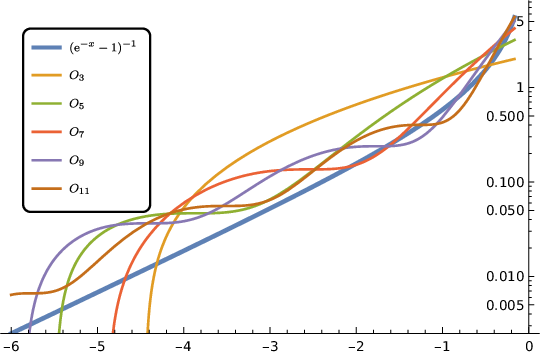}}
\caption{Taylor $T_{2K+1}$ and optimized $O_{2K+1}$ approximations of the Planckian $\beta_{BE}$.}
\label{fig:comparison_be}
\end{figure}

The $\mathcal{L}^2(-L,-1/L)$ error of the $O_{2K+1}$ approximation for different $K$ and $L$ is given in Fig.~\ref{fig:error_be}, where we can again observe the spectral convergence with respect to $K$ for fixed $L$, and the convergence rates are lower for larger intervals. Comparing Fig.~\ref{fig:error_be} with Fig.~\ref{fig:error_bs}, we can find that the $\mathcal{L}^2$ error for the Bose-Einstein entropy is significantly larger. This is likely due to the singularity of the Planckian at zero. No polynomial possesses the same property, making the function more difficult to approximate using polynomials. The exponential function, however, tends to infinity only when $x$ tends to infinity, which all polynomials with positive leading coefficients also satisfy.

\begin{figure}[!ht]
\centering
\includegraphics[width=.8\textwidth]{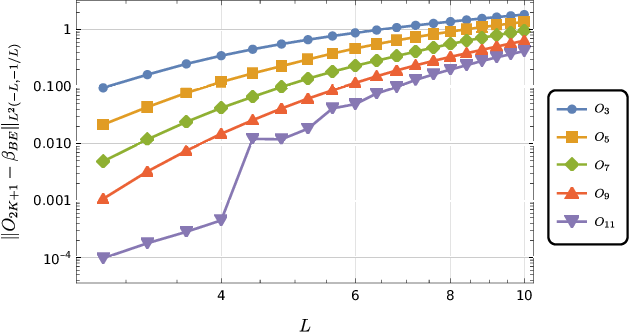}
\caption{Error plot for the $O_{2K+1}$ approximations of the exponential function $\beta_{BS}$}
\label{fig:error_be}
\end{figure}

\section{Numerical approximation of specific distributions}
\label{sec:numerics}
To complete this study, we reproduce simulations from~\cite{PhiRTE1} with the different approximations and compare their results. These consist in approximating distributions that correspond to physical regime, namely a near-beam distribution, a distribution corresponding to two beams crossing each others, and smooth distributions.    

Finding the approximation requires solving the moment inversion problem:
\begin{equation}
\label{eq:mnt_inv}
  \int_{\mathbb{S}^2} \mathbf{m}(\Omega) \beta(\boldsymbol{\lambda}^T \mathbf{m}(\Omega)) \,\mathrm{d}\Omega = \mathbf{U},
\end{equation}
for a given vector of moments $\mathbf{U}$ of specific distributions. The vector function $\mathbf{m}(\Omega)$ is chosen as all real spherical harmonics up to a certain degree $N$. Here the function $\beta(\cdot)$ is either the $\beta_{2K+1}$, $T_{2K+1}$ or $O_{2K+1}$ approximation of the exponential function $\beta_{BS}$, and either the $T_{2K+1}$ or $O_{2K+1}$ approximation of the Planckian $\beta_{BE}$. The right-hand side $\mathbf{U}$ is given by the moments of a given function, which means we first choose a function $I(\Omega)$, and then set
\begin{displaymath}
  \mathbf{U} = \int_{\mathbb{S}^2} \mathbf{m}(\Omega) I(\Omega) \,\mathrm{d}\Omega.
\end{displaymath}
After solving $\boldsymbol{\lambda}$ from \eqref{eq:mnt_inv}, the function $\beta(\boldsymbol{\lambda}^T \mathbf{m}(\Omega))$ is regarded as an approximation of $I(\Omega)$. For clarification, we will add the subscript $N$ to the name of the model to denote the moment method, the first subscript $N$ refers to the moment order and the second $2K+1$ to the degree of the polynomial approximation $\beta_{2K+1}$, $T_{2K+1}$ or $O_{2K+1}$. For example, if we use spherical harmonics up to degree $N$ in $\Omega$ and choose $\beta(\cdot)$ to be $T_{2K+1}$, the model is denoted as $T_{N,2K+1}$. Similarly, we will also consider the $\beta_{N,2K+1}$ and $O_{N,2K+1}$ models below.

The equation \eqref{eq:mnt_inv} is solved by Newton's method, for which we need to compute the Jacobian
\begin{displaymath}
\int_{\mathbb{S}^2} \mathbf{m}(\Omega) [\mathbf{m}(\Omega)]^T \beta'(\boldsymbol{\lambda}^T \mathbf{m}(\Omega)) \,\mathrm{d}\Omega.
\end{displaymath}
Since the integrand is a polynomial of $\Omega$, the integral can be computed exactly using appropriate integration formulas. Here we adopt the Lebedev quadrature (\cite{Lebedev1976quadratures, Lebedev1999quadrature}) as in \cite{PhiRTE1}. The number of quadrature points is chosen such that the degree of the quadrature is no less than the degree of the polynomial.

\subsection{Single beam approximation}
In this section, we apply these approximate entropy models in the approximation of a single beam. Since all the models are rotationally invariant, the direction of the beam does not affect the result. For simplicity, we consider the approximation of the Dirac-delta function
\begin{displaymath}
I(\Omega) = \delta(\Omega - \Omega_0),
\end{displaymath}
where $\Omega_0 = (0,0,1)^T$. Some approximations for the Boltzmann-Shannon entropy with $N = 1$ and $K = 2$ are given in Fig.~\ref{fig:single_beam_M1_BS}. The plots show that the $\beta_{1,5}$ model gives a remarkably better result than the $T_{1,5}$. This is not surprising because the value of the approximate function is all below $1.0$, which corresponds to $\exp(x)$ with a negative $x$, where the $\beta_{1,5}$ model can give a better approximation. The quality of the $O_{1,5}$ model shown in Fig.~\ref{fig:single_beam_M1_BS_O5} then lies in-between, since it approximates the exponential on the interval $[-5,5]$, which balances both the positive and the negative parts. In order to improve the result, we can shift the domain to the negative side as in Fig.~\ref{fig:single_beam_M1_BS_O10}, so that a result similar to the $\beta_{1,5}$ model can be obtained.

\begin{figure}[h!]
\centering
\subfloat[$\beta_{1,5}$]{\includegraphics[width=.45\textwidth]{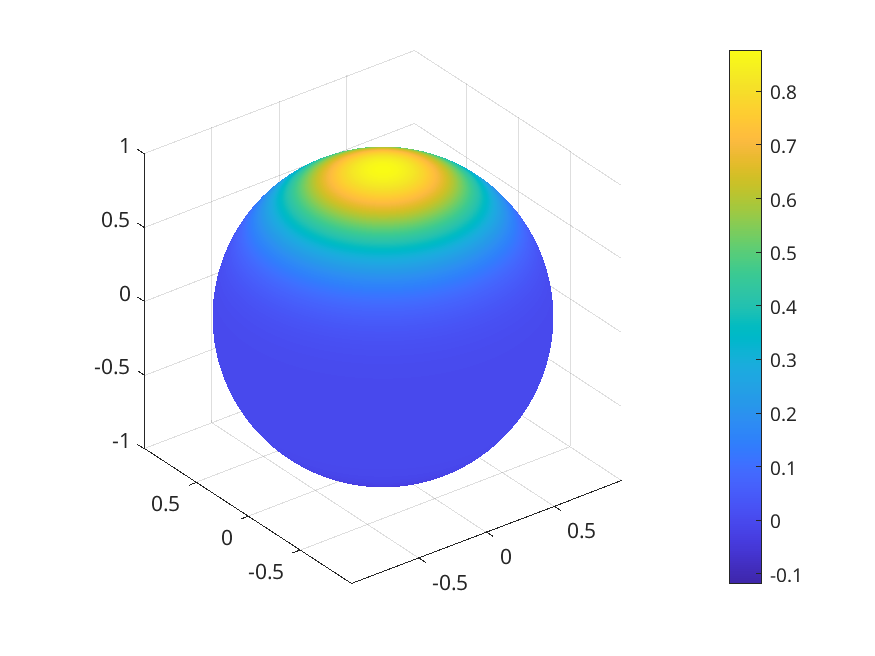}} \qquad
\subfloat[$T_{1,5}$, B-S entropy, $x_0 = 0$]{\includegraphics[width=.45\textwidth]{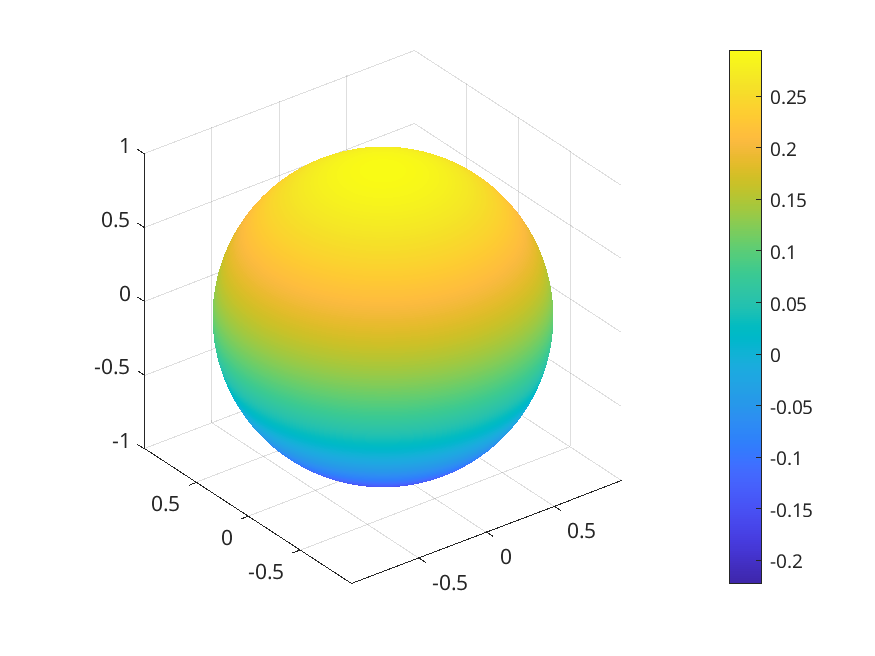}} \\
\subfloat[$O_{1,5}$, B-S entropy, {$[-5,5]$}]{\label{fig:single_beam_M1_BS_O5}\includegraphics[width=.45\textwidth]{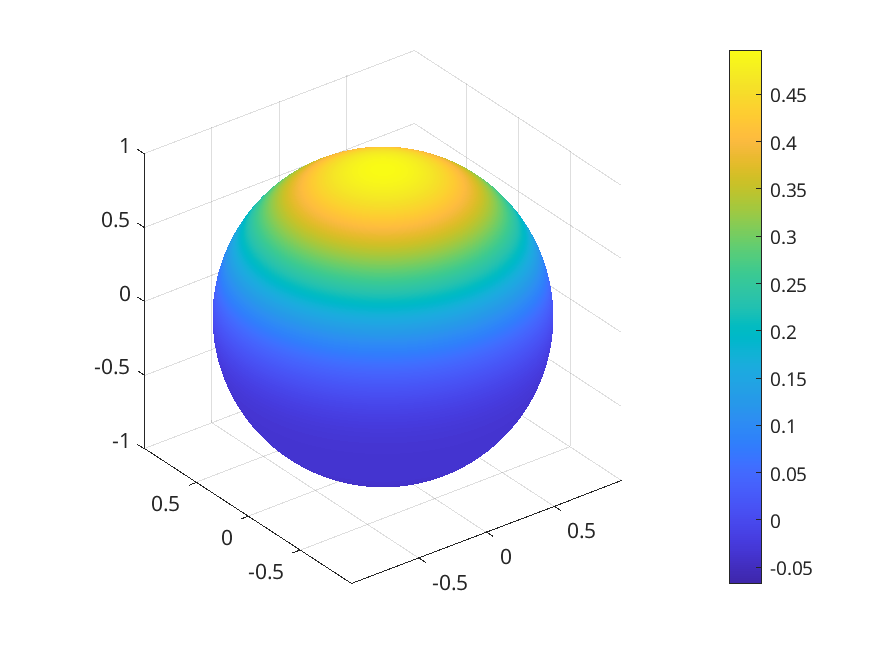}} \qquad
\subfloat[$O_{1,5}$, B-S entropy, {$[-10,0]$}]{\label{fig:single_beam_M1_BS_O10}\includegraphics[width=.45\textwidth]{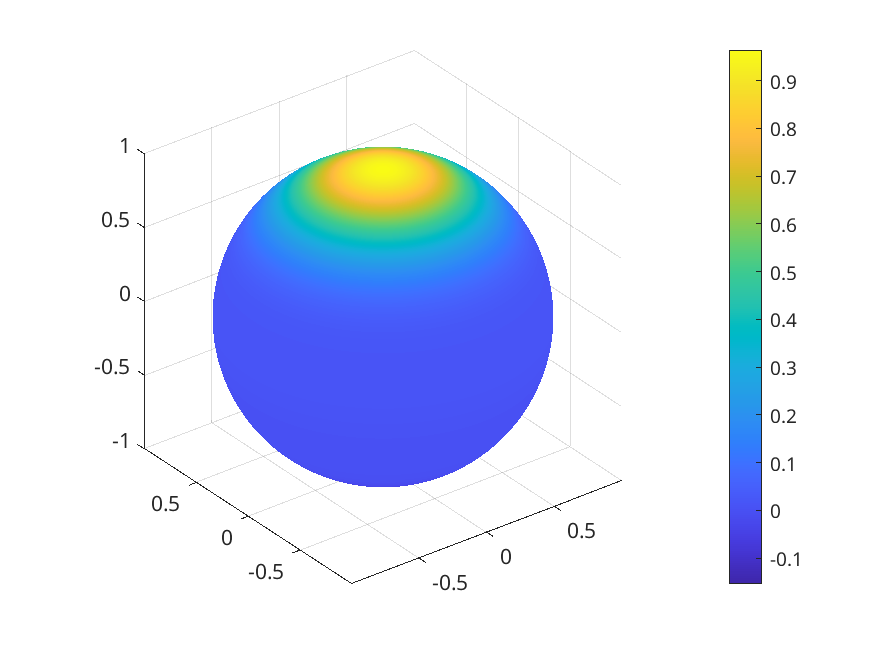}}
\caption{Approximation of a single beam using different models approximating the result of the maximum Boltzmann-Shannon entropy.}
\label{fig:single_beam_M1_BS}
\end{figure}

Similar phenomena can be observed when the approximation of the Bose-Einstein entropy is applied. The results are plotted in Fig.~\ref{fig:single_beam_M1_BE}. For both $T_{1,2K+1}$ and $O_{1,2K+1}$ models, the approximate intensity function is closer to $I(\Omega)$ if the parameters are chosen to fit the range of the function values. 
\begin{figure}[h!]
\centering
\subfloat[$T_{1,5}$, B-E entropy, $x_0 = -2.6$]{\includegraphics[width=.45\textwidth]{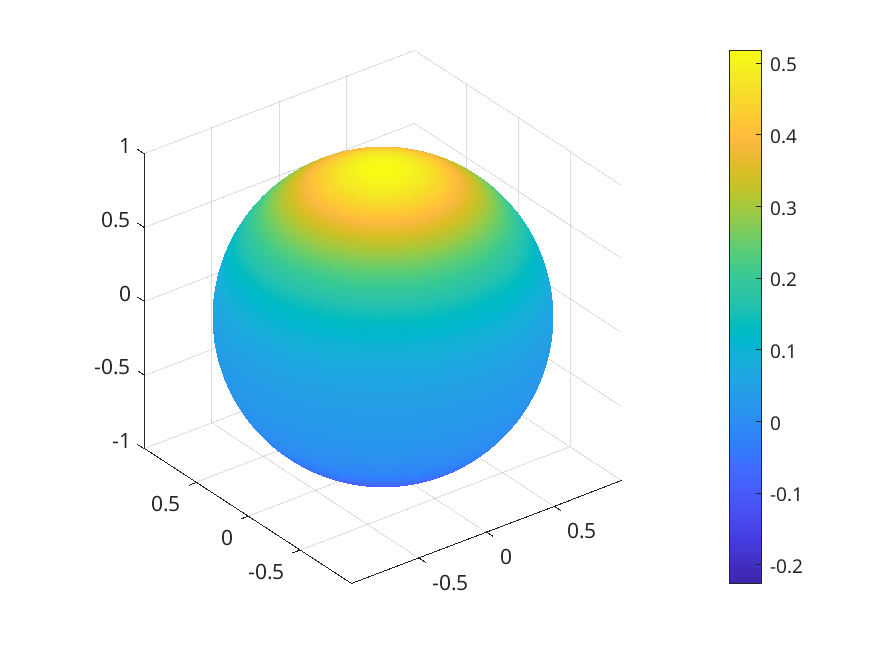}} \qquad
\subfloat[$T_{1,5}$, B-E entropy, $x_0= -5.5$]{\includegraphics[width=.45\textwidth]{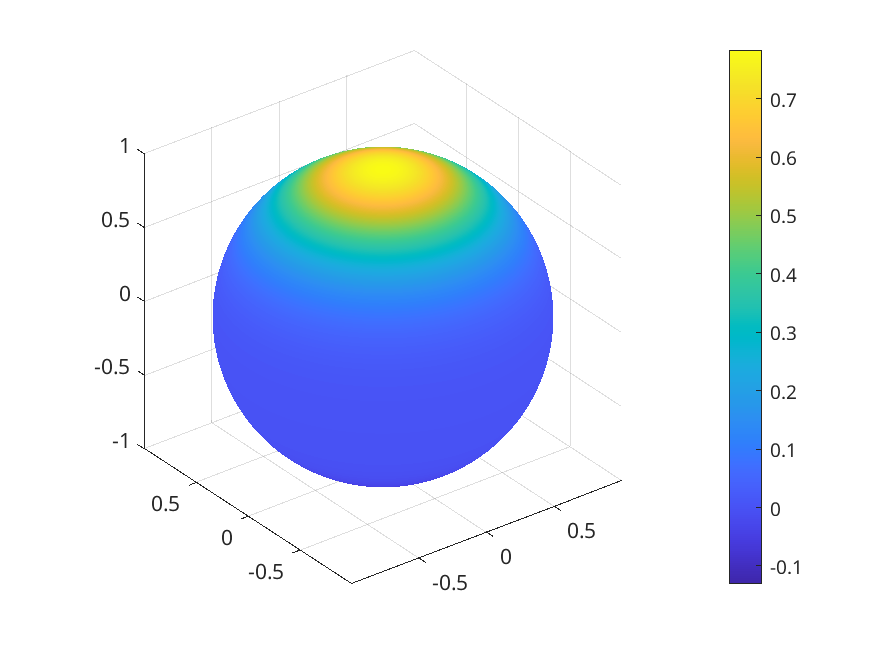}} \\
\subfloat[$O_{1,5}$, B-E entropy, {$[-5,-0.2]$}]{\label{fig:single_beam_M1_BE_O5}\includegraphics[width=.45\textwidth]{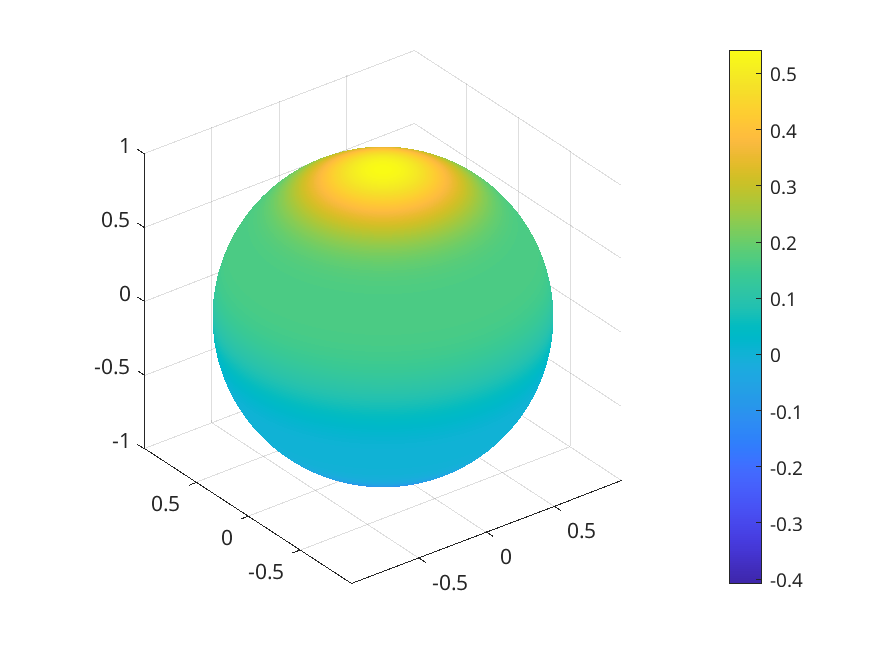}} \qquad
\subfloat[$O_{1,5}$, B-E entropy, {$[-10,-1]$}]{\label{fig:single_beam_M1_BE_O10}\includegraphics[width=.45\textwidth]{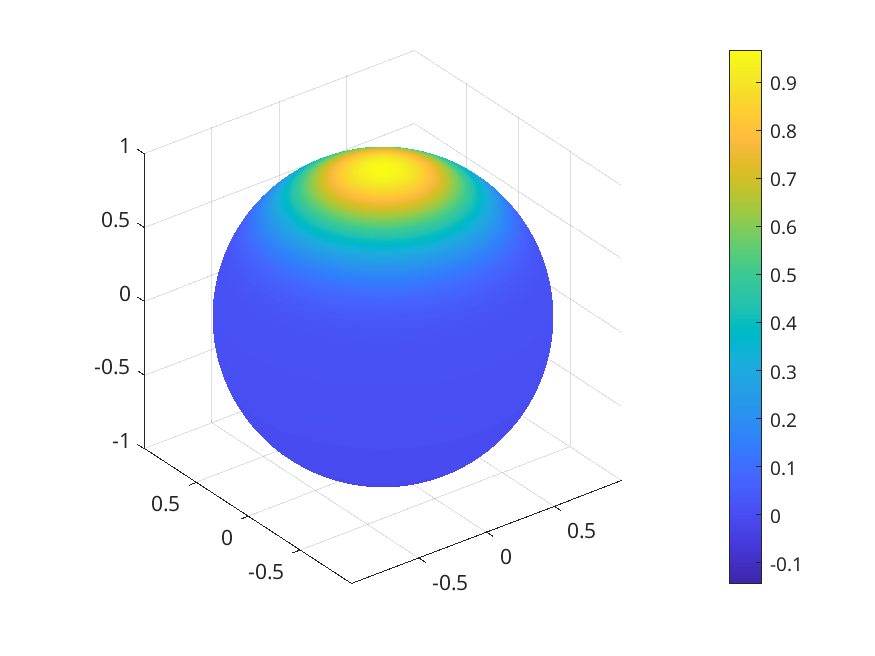}}
\caption{Approximation of a single beam using different models approximating the result of the maximum Bose-Einstein entropy.}
\label{fig:single_beam_M1_BE}
\end{figure}

All the results above are only for $N = 1$ and $K = 2$. To improve the results, we can increase either $N$ or $K$. Fig.~\ref{fig:single_beam_M1_K6} includes some results for $K = 6$. Increasing the value of $K$ from $2$ to $6$ does provide improved result for all other parameters, but the beam is still widely spread for all cases. A more efficient way to get improvement is to increase $N$ from $1$ to $3$. The results shown in Fig.~\ref{fig:single_beam_M3} exhibit much sharper beams compared with all previous results, and the functions are mostly positive except the Taylor model with $x_0 = 0$. In this test case, the optimized model shows the highest peak value for both types of entropy. Meanwhile, all these results show that the $\beta_{N,2K+1}$ model studied in \cite{PhiRTE1} is also a good choice for problems involving beams when the Boltzmann-Shannon entropy is considered. However, this model does not have a counterpart for the Bose-Einstein entropy.

\begin{figure}[h!]
\centering
\subfloat[$\beta_{1,13}$]{\includegraphics[width=.45\textwidth]{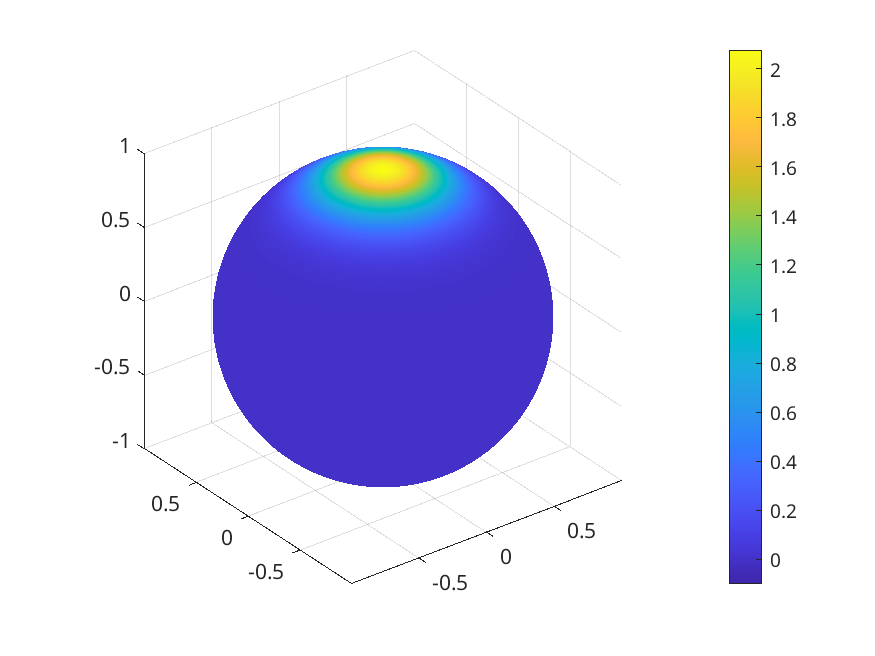}} \qquad
\subfloat[$T_{1,13}$, B-S entropy, $x_0= 0$]{\includegraphics[width=.45\textwidth]{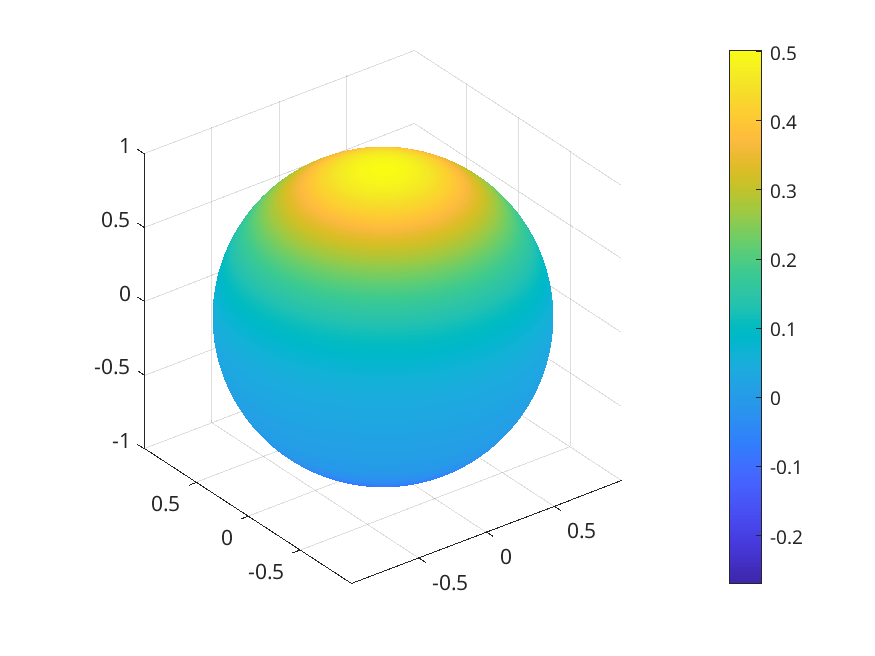}} \\
\subfloat[$T_{1,13}$, B-S entropy, $x_0 = -5$]{\includegraphics[width=.45\textwidth]{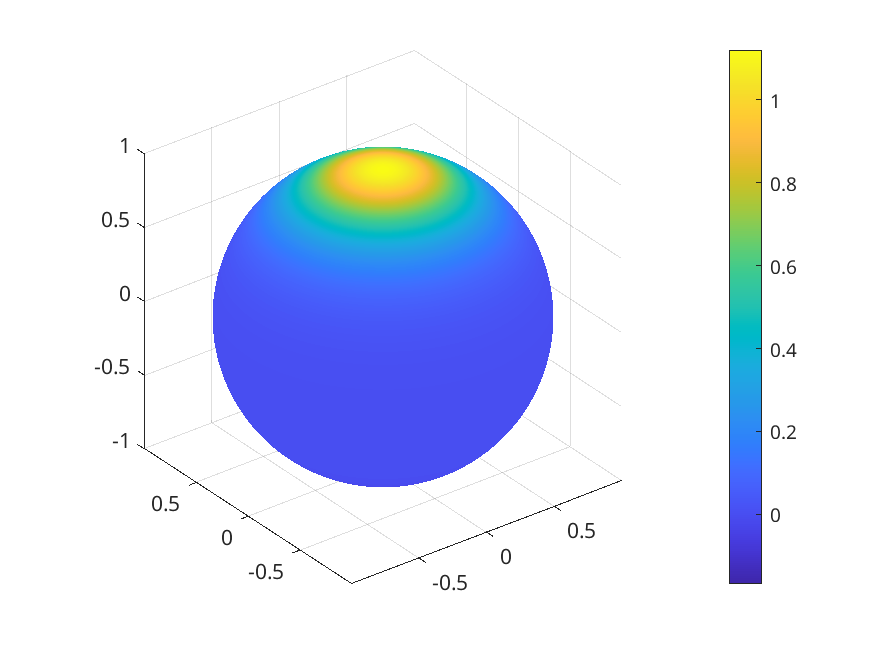}} \qquad
\subfloat[$O_{1,13}$, B-S entropy, {$[-10,0]$}]{\includegraphics[width=.45\textwidth]{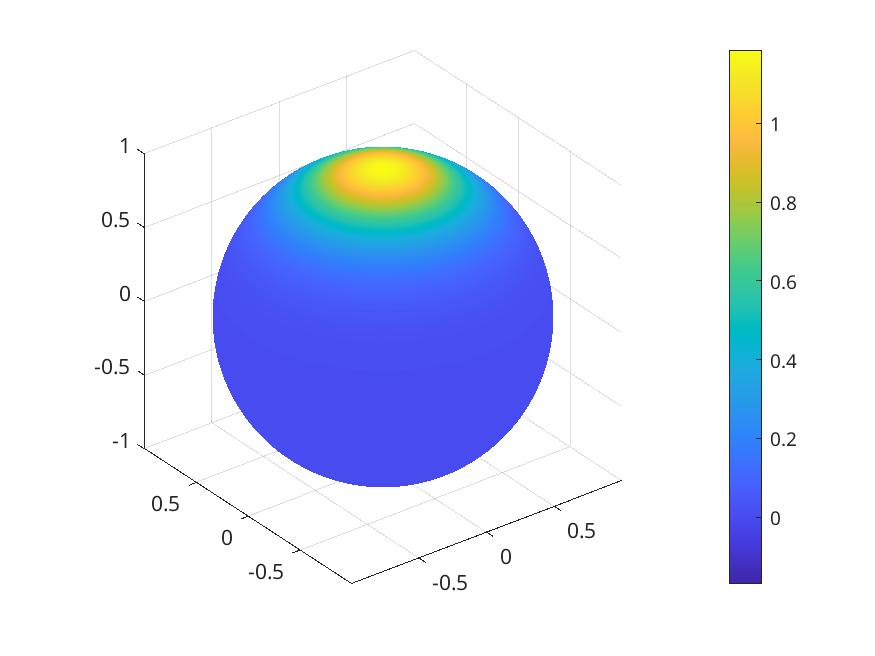}} \\
\subfloat[$T_{1,13}$, B-E entropy, $x_0 = -5.5$]{\includegraphics[width=.45\textwidth]{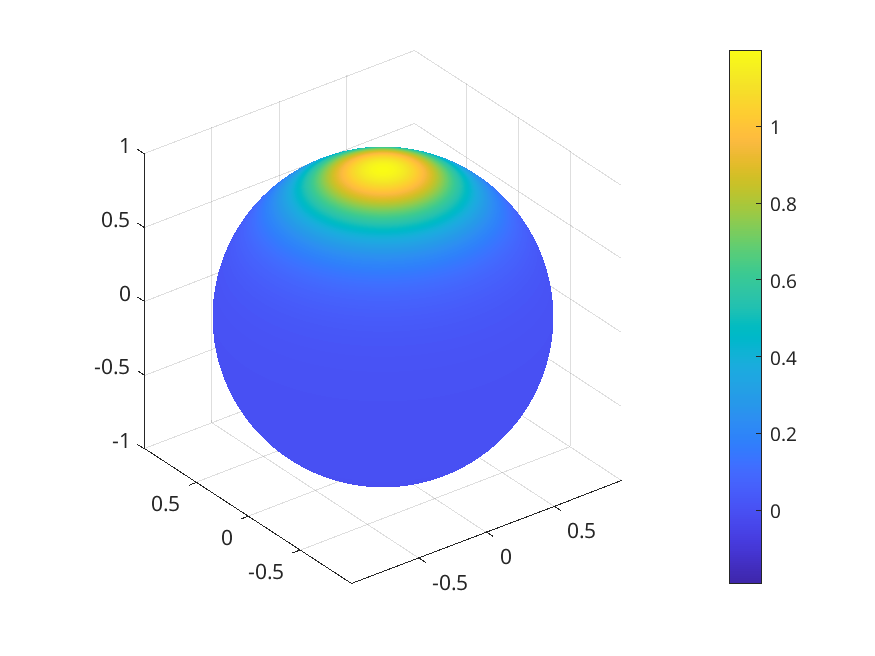}} \qquad
\subfloat[$O_{1,13}$, B-E entropy, {$[-10,-1]$}]{\includegraphics[width=.45\textwidth]{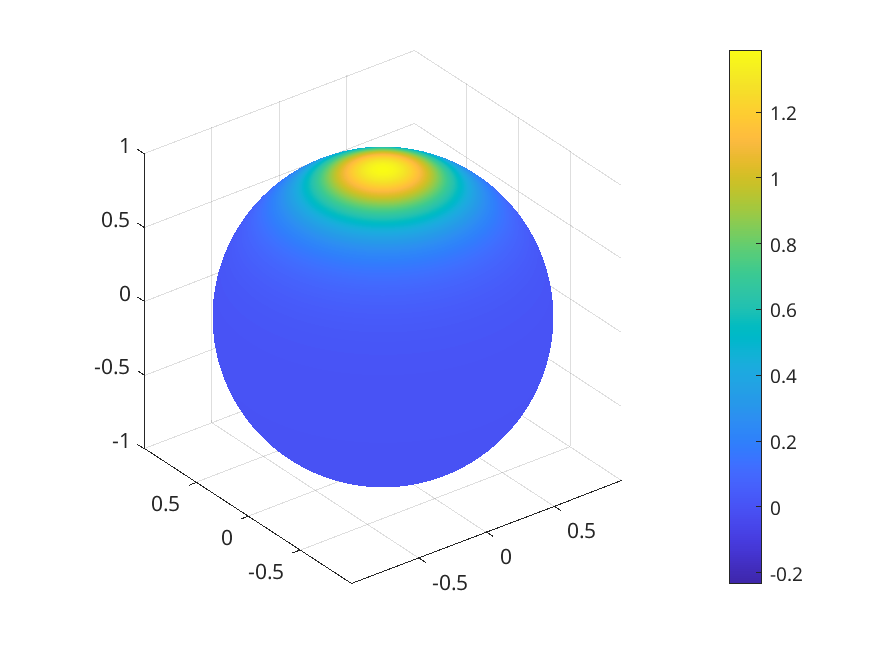}}
\caption{Approximation of a single beam using different models}
\label{fig:single_beam_M1_K6}
\end{figure}

\begin{figure}[h!]
\centering
\subfloat[$\beta_{3,5}$]{\includegraphics[width=.45\textwidth]{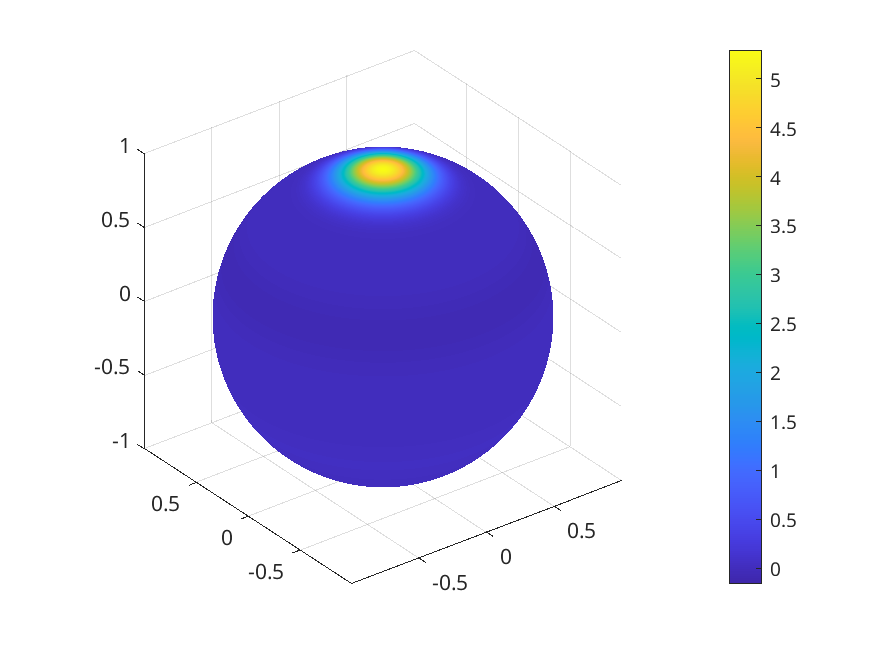}} \qquad
\subfloat[$T_{3,5}$, B-S entropy, $x_0= 0$]{\includegraphics[width=.45\textwidth]{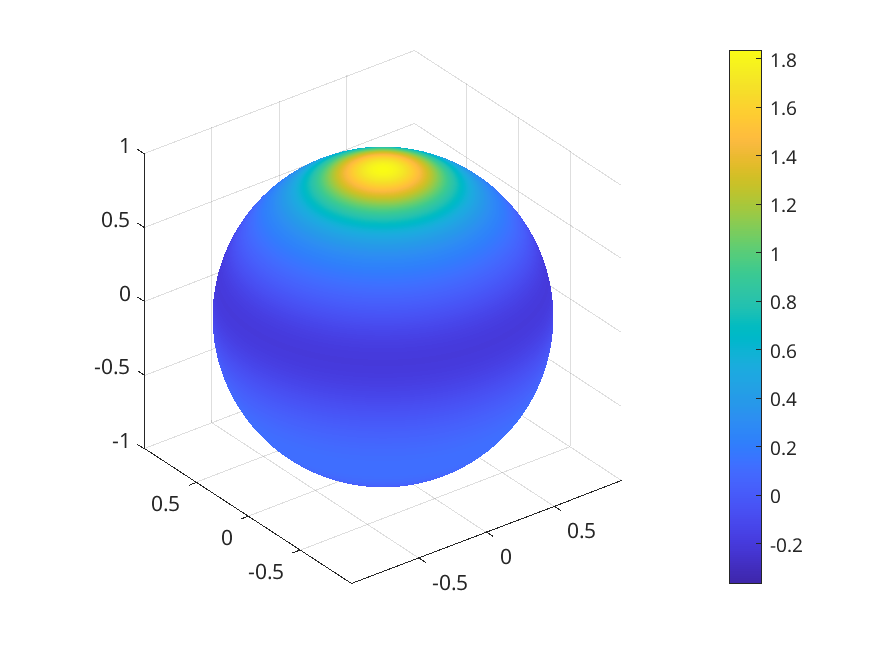}} \\
\subfloat[$T_{3,5}$, B-S entropy, $x_0 = -5$]{\includegraphics[width=.45\textwidth]{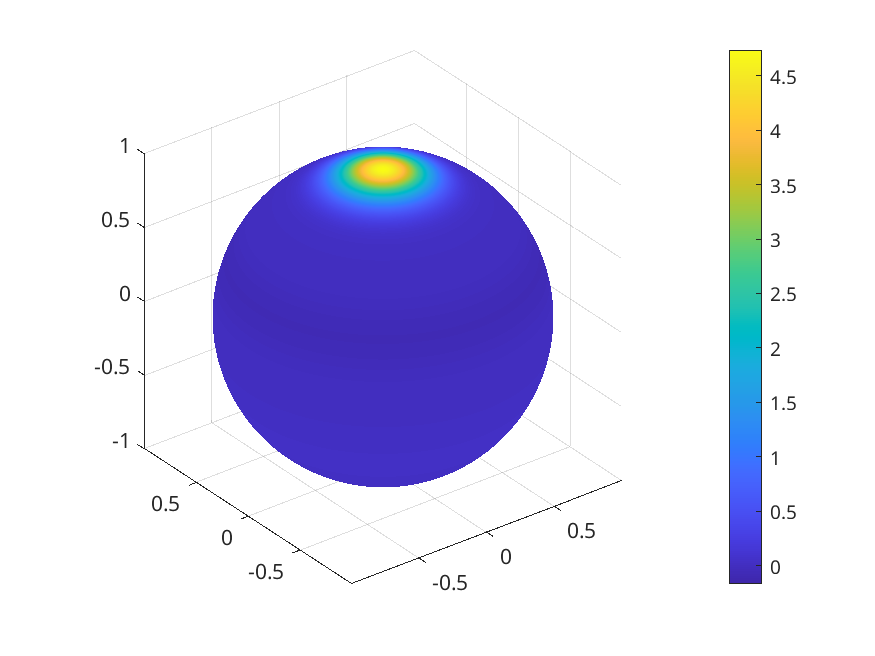}} \qquad
\subfloat[$O_{3,5}$, B-S entropy, {$[-10,0]$}]{\includegraphics[width=.45\textwidth]{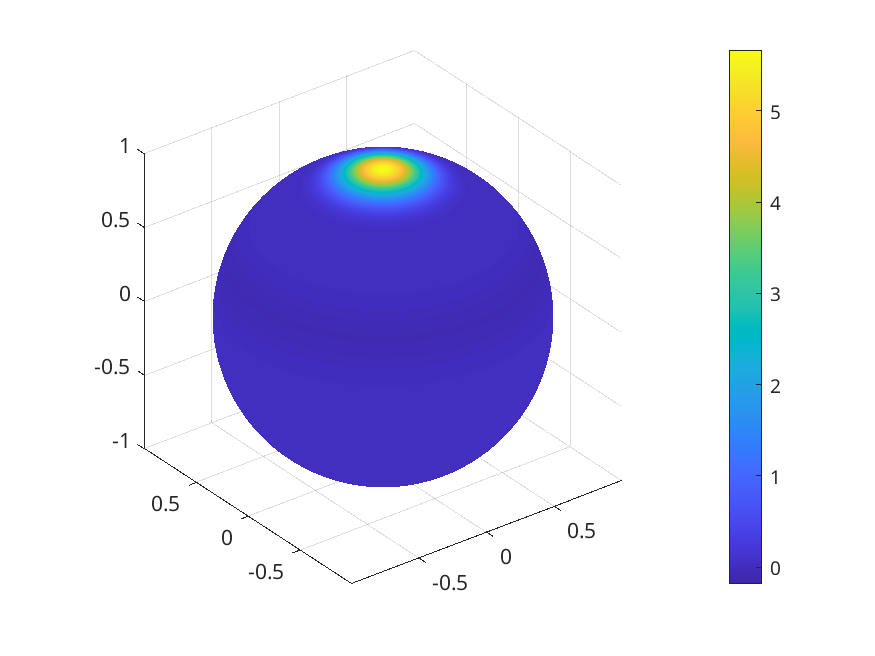}} \\
\subfloat[$T_{3,5}$, B-E entropy, $x_0 = -5.5$]{\includegraphics[width=.45\textwidth]{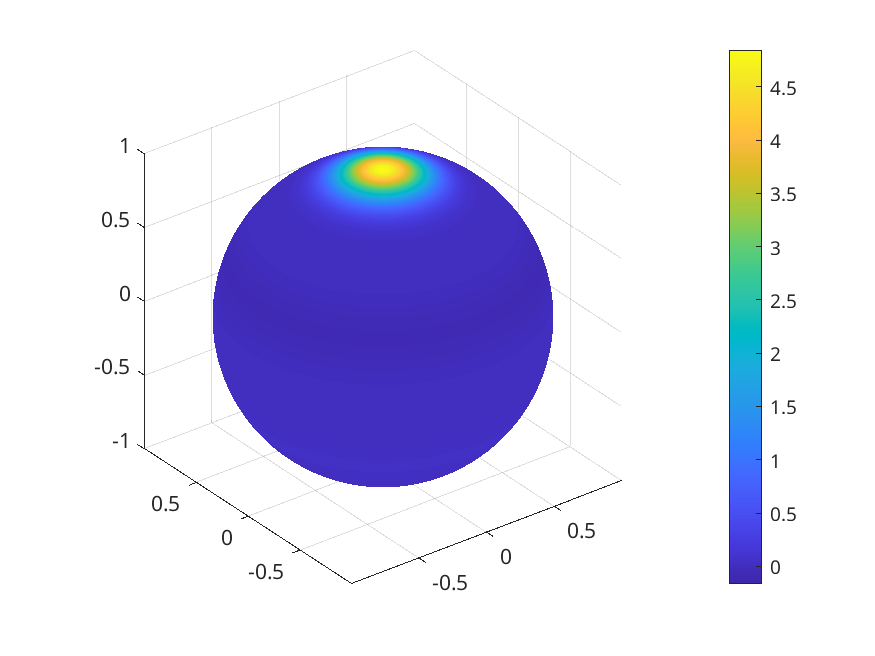}} \qquad
\subfloat[$O_{3,5}$, B-E entropy, {$[-10,-1]$}]{\includegraphics[width=.45\textwidth]{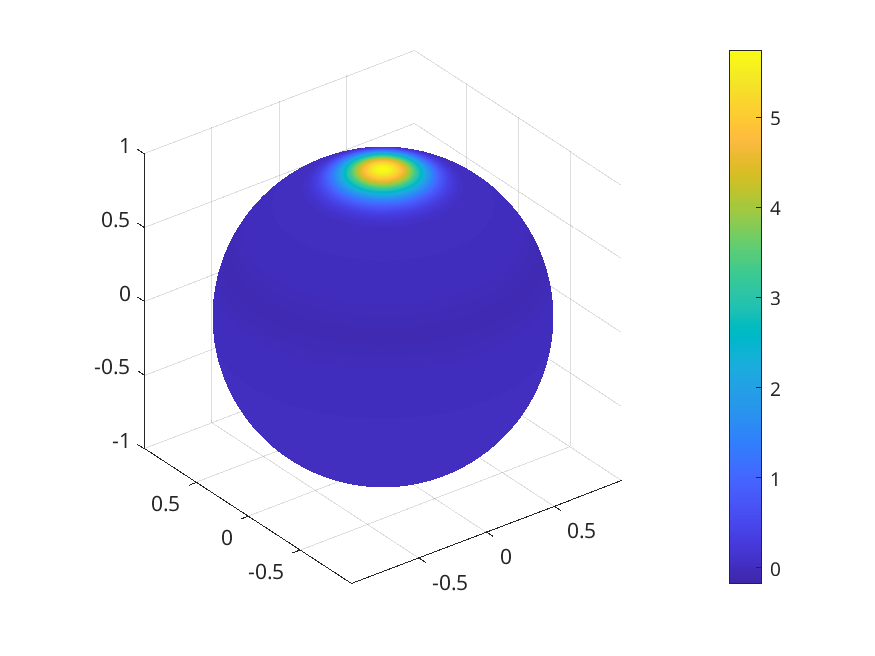}}
\caption{Approximation of a single beam using different models}
\label{fig:single_beam_M3}
\end{figure}

\subsection{Double beam approximation}
We now considers the approximation of the following function:
\begin{displaymath}
  I(\Omega) = \delta(\Omega - \Omega_1) + \delta(\Omega - \Omega_2),
\end{displaymath}
which often occurs when two beams cross each other. Here we focus only on the $T_{N,2K+1}$ and $O_{N,2K+1}$ models, and we refer the readers to \cite{PhiRTE1} for results of the $\beta_{N,2K+1}$ models.

Since there are two beams in the intensity function, a model with $N = 1$ cannot give a meaningful approximation. In our experiments, we choose $\Omega_1 = (0,0,1)^T$, $\Omega_2 = (1,0,0)^T$ and use $N = 3$ and $N = 9$ in the approximation. Other parameters are chosen to be the better combination in the previous subsection. In particular, the value of $K$ is fixed to be $2$ in all our examples.

For $N = 3$ (see Fig.~\ref{fig:double_beam_M3}), the two Taylor models give quite similar results. The two beams are correctly detected with correction locations, and they are both smeared out due to the smooth approximation. The two optimized results provide sharper beams, as the peak value of the distribution is higher. Negative values can still be spotted near the point $(0,-1,0)^T$, which can be improved by increasing $N$ or $K$. Here we only perform experiments with $N = 9$, which can be found in Fig.~\ref{fig:double_beam_M9}. The two bright spots are much more pointy than the results of $N = 3$, and the peak values are now significantly higher. Again, the $O_{N,2K+1}$ models perform slightly better than the $T_{N,2K+1}$ models for both types of entropy.

\begin{figure}[h!]
\centering
\subfloat[$T_{3,5}$, B-S entropy, $x_0 = -5$]{\includegraphics[width=.45\textwidth]{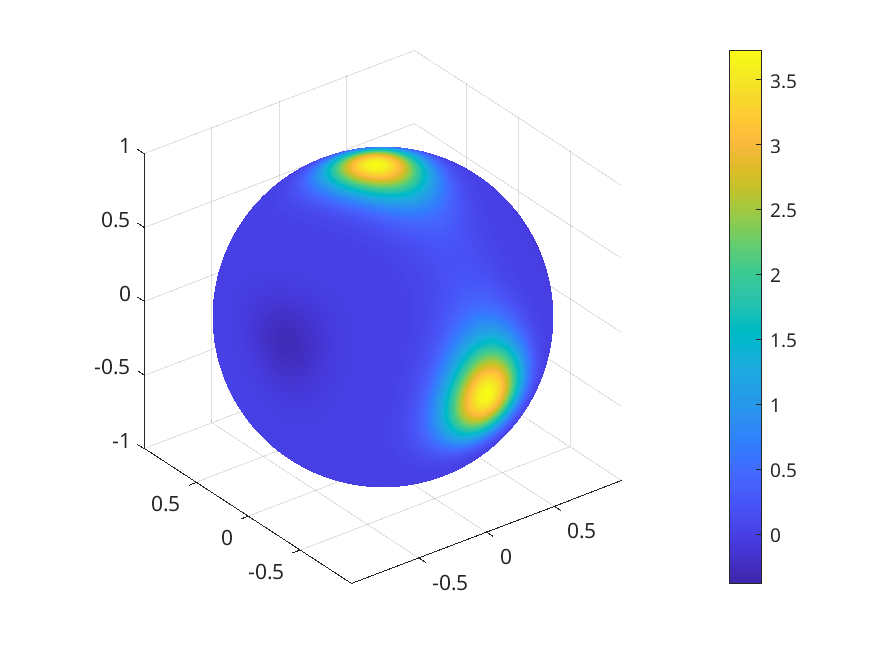}} \qquad
\subfloat[$O_{3,5}$, B-S entropy, {$[-10,0]$}]{\includegraphics[width=.45\textwidth]{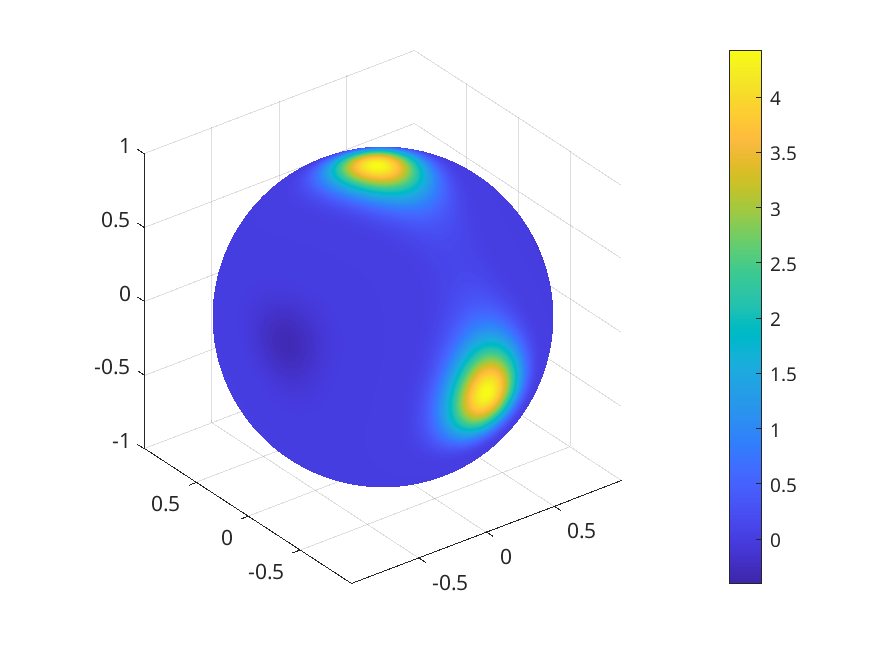}} \\
\subfloat[$T_{3,5}$, B-E entropy, $x_0 = -5.5$]{\includegraphics[width=.45\textwidth]{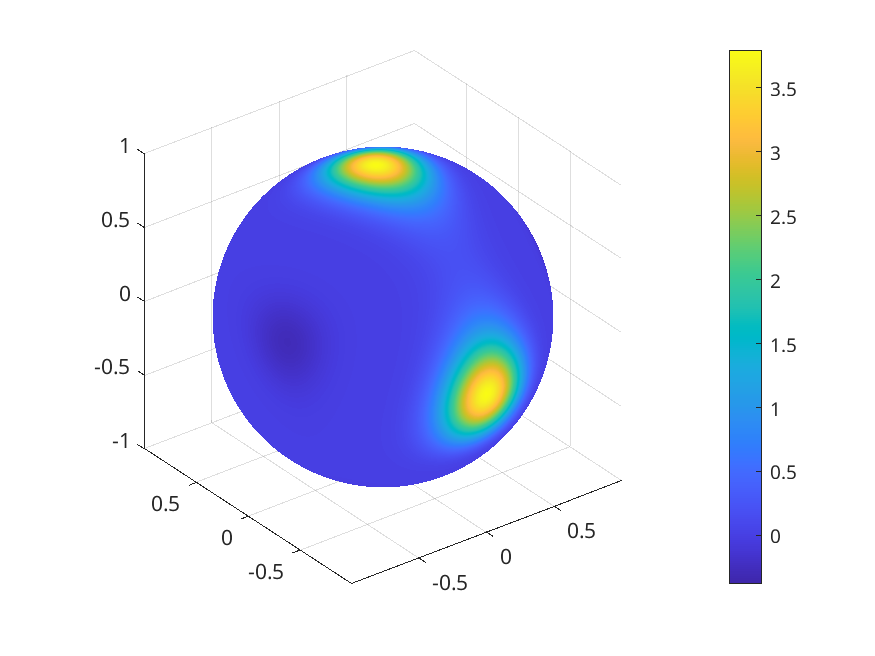}} \qquad
\subfloat[$O_{3,5}$, B-E entropy, {$[-10,-1]$}]{\includegraphics[width=.45\textwidth]{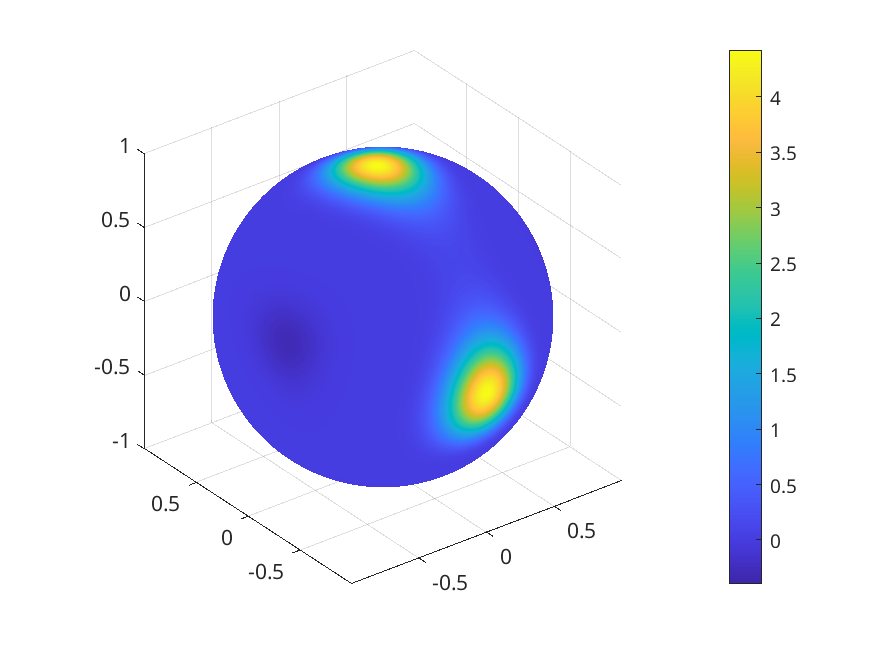}}
\caption{Approximation of a double-beam function using different models.}
\label{fig:double_beam_M3}
\end{figure}

\begin{figure}[h!]
\centering
\subfloat[$T_{9,5}$, B-S entropy, $x_0 = -5$]{\includegraphics[width=.45\textwidth]{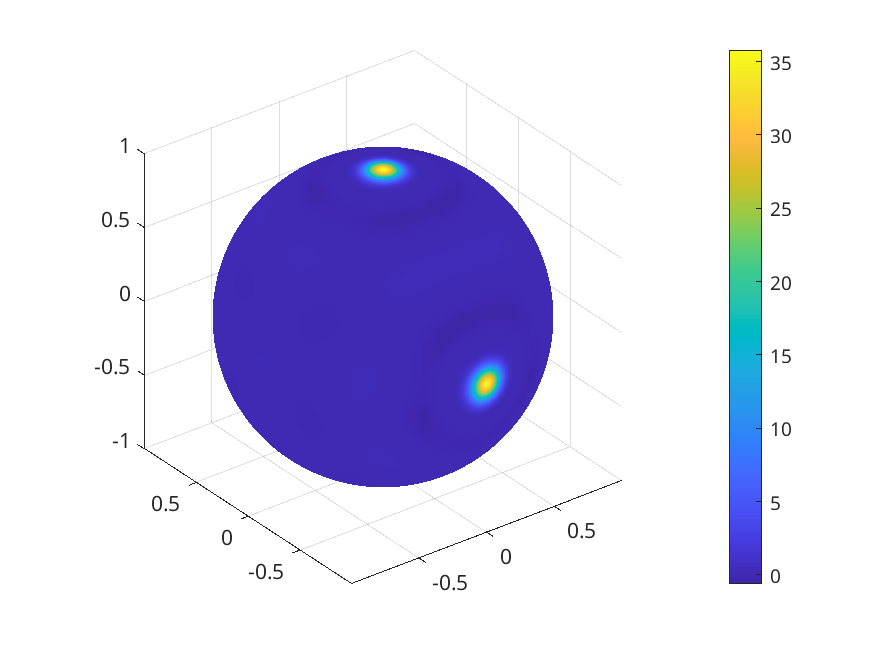}} \qquad
\subfloat[$O_{9,5}$, B-S entropy, {$[-10,0]$}]{\includegraphics[width=.45\textwidth]{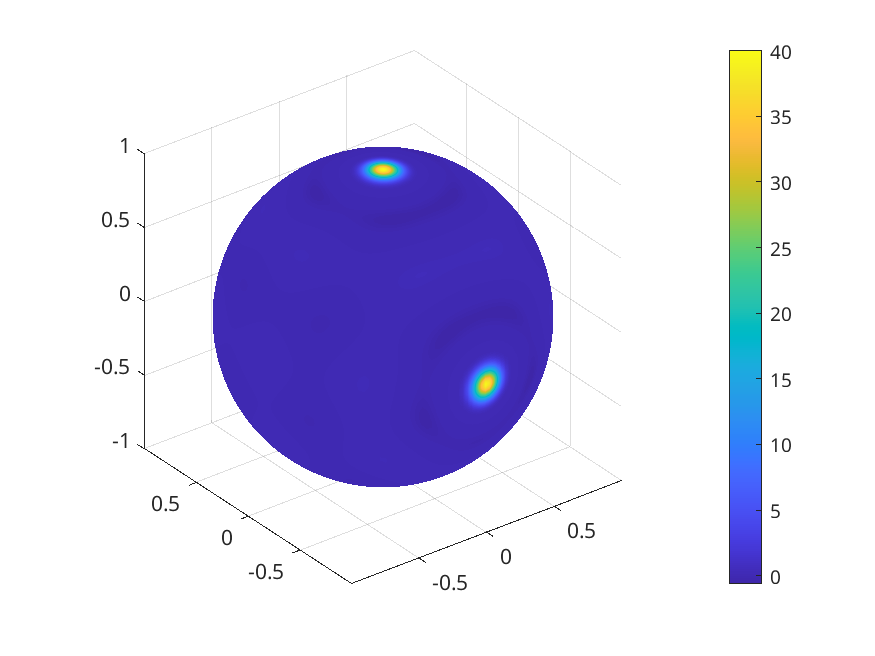}} \\
\subfloat[$T_{9,5}$, B-E entropy, $x_0 = -5.5$]{\includegraphics[width=.45\textwidth]{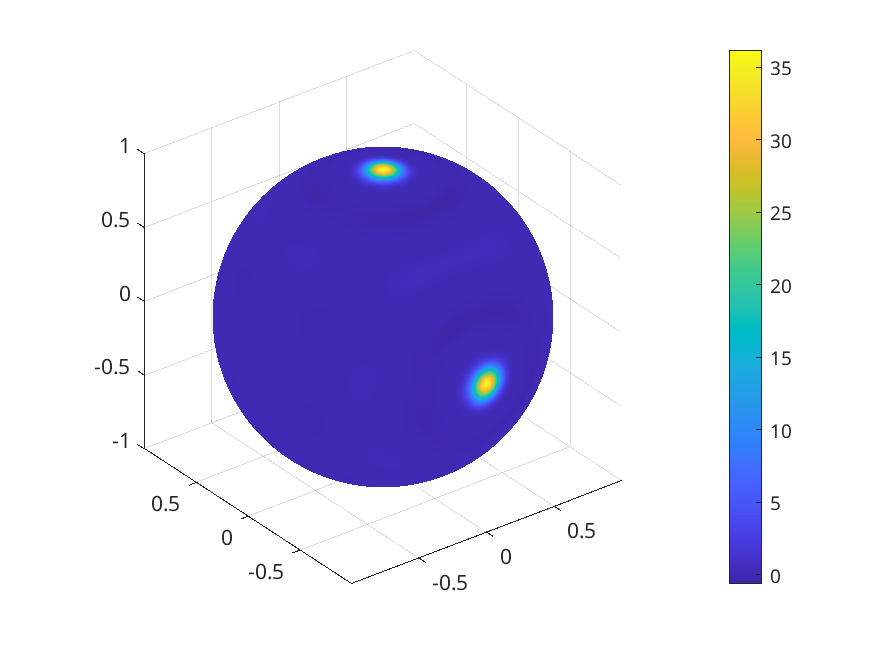}} \qquad
\subfloat[$O_{9,5}$, B-E entropy, {$[-10,-1]$}]{\includegraphics[width=.45\textwidth]{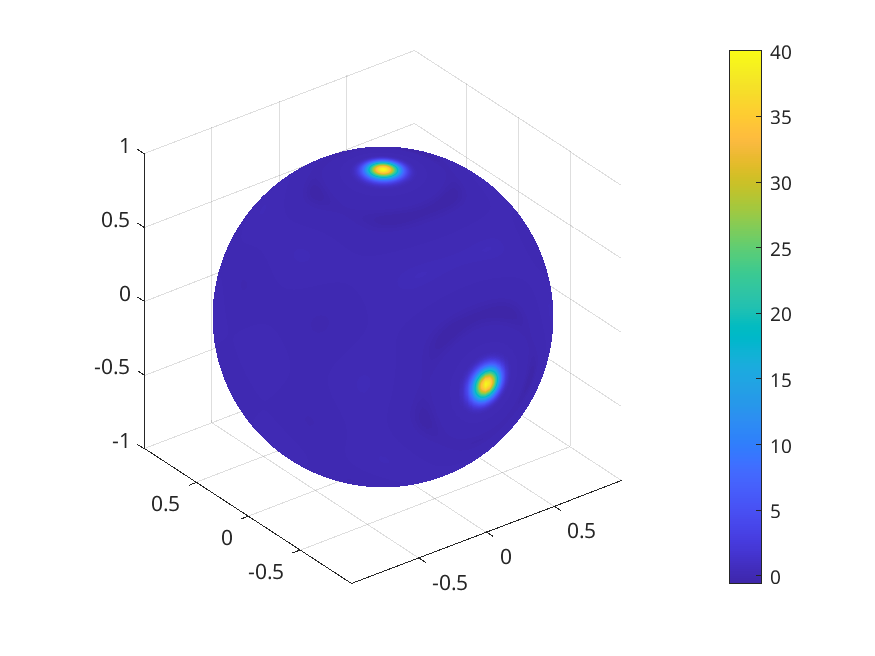}}
\caption{Approximation of a double-beam function using different models.}
\label{fig:double_beam_M9}
\end{figure}

\subsection{Approximating smooth functions}
We now consider the the approximation of smooth functions and hope to observe spectral convergence. We take the six-Gaussian function considered in \cite{PhiRTE1}:
\begin{displaymath}
  I(\Omega) = \sum_{k=1}^6 \exp\left(-5\|\Omega - \Omega_k\|^2\right),
\end{displaymath}
where
\begin{gather*}
  \Omega_1 = (1,0,0)^T, \qquad \Omega_2 = (-1,0,0)^T, \qquad \Omega_3 = (0,1,0)^T, \\
  \Omega_4 = (0,-1,0)^T, \qquad \Omega_5 = (0,0,1)^T, \qquad \Omega_6 = (0,0,-1)^T.
\end{gather*}
This function and its approximation using the $\beta_{5,5}$ model are plotted in Fig.~\ref{fig:six_Gaussian}. It can be seen that the $\beta_{5,5}$
approximation overestimates the peak value. Fig.~\ref{fig:six_Gaussian_BS} shows some approximations based on the Boltzmann-Shannon entropy. It can be seen by naked eyes that the $O_{5,5}$ model defined by optimization on $[-5,5]$ gives the best result.

\begin{figure}[h!]
\centering
\subfloat[Six-Gaussian function]{\includegraphics[width=.45\textwidth]{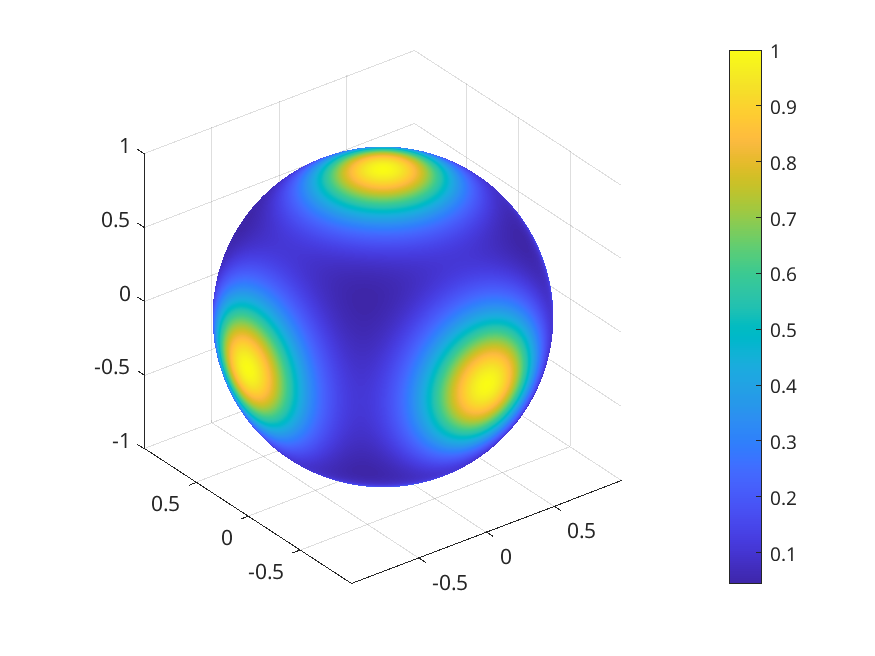}} \qquad
\subfloat[$\beta_{5,5}$ approximation]{\includegraphics[width=.45\textwidth]{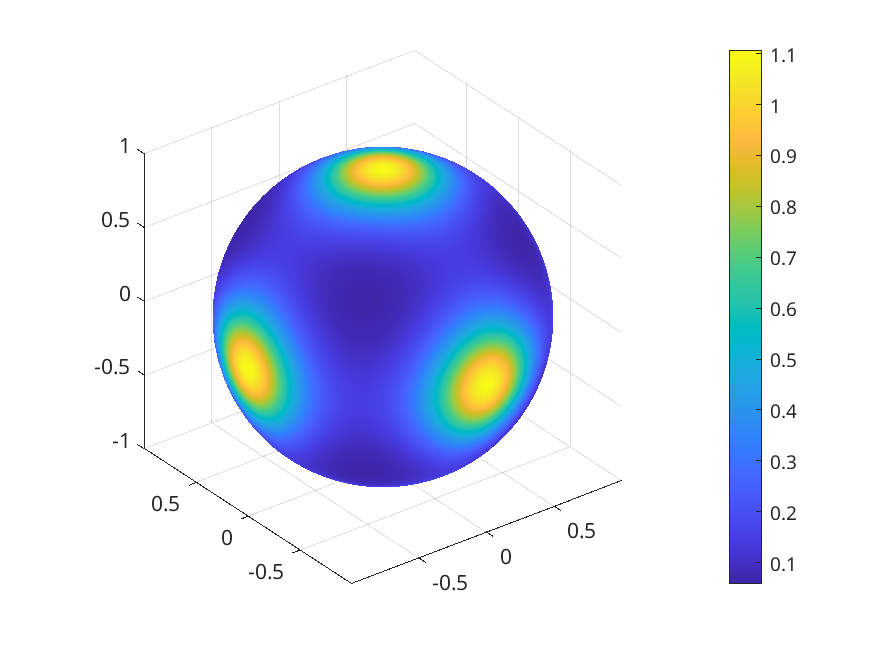}}
\caption{The six-Gaussian function and its $\beta_{5,5}$ approximation.}
\label{fig:six_Gaussian}
\end{figure}

\begin{figure}[h!]
\centering
\subfloat[$T_{5,5}$, B-S entropy, {$x_0 = 0$}]{\includegraphics[width=.45\textwidth]{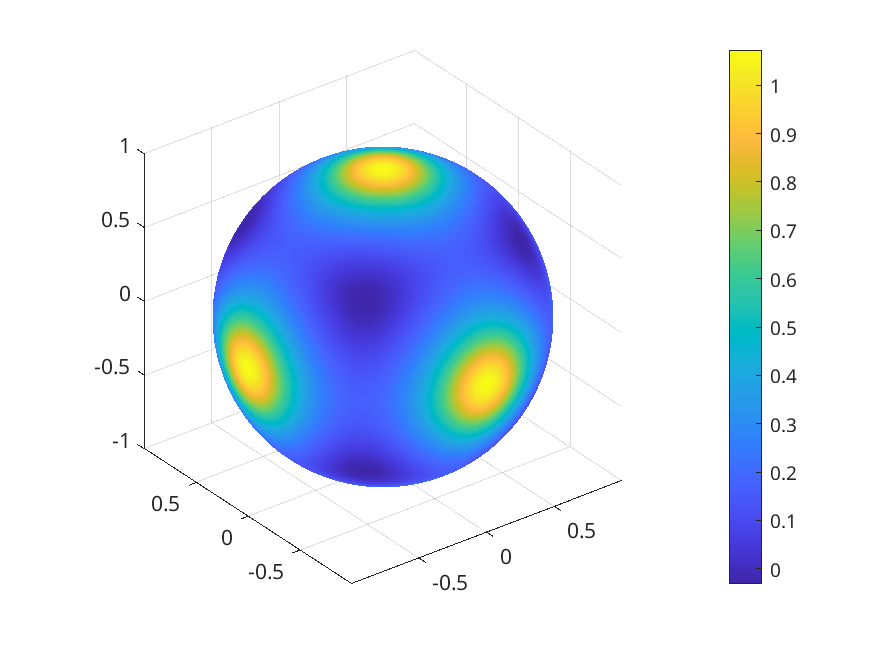}} \qquad
\subfloat[$O_{5,5}$, B-S entropy, {$[-5,5]$}]{\includegraphics[width=.45\textwidth]{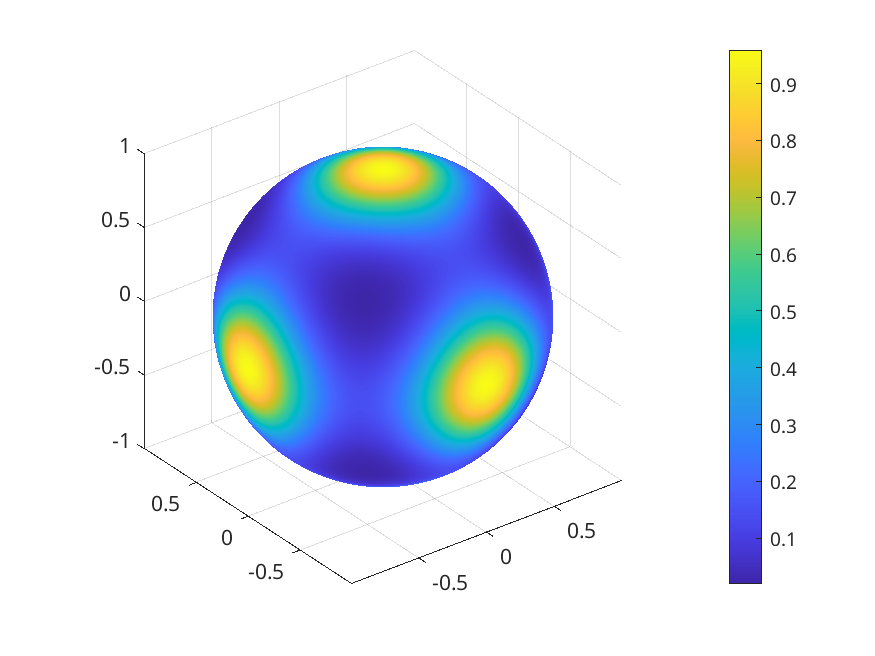}} \\
\subfloat[$T_{5,5}$, B-S entropy, $x_0 = -5$]{\includegraphics[width=.45\textwidth]{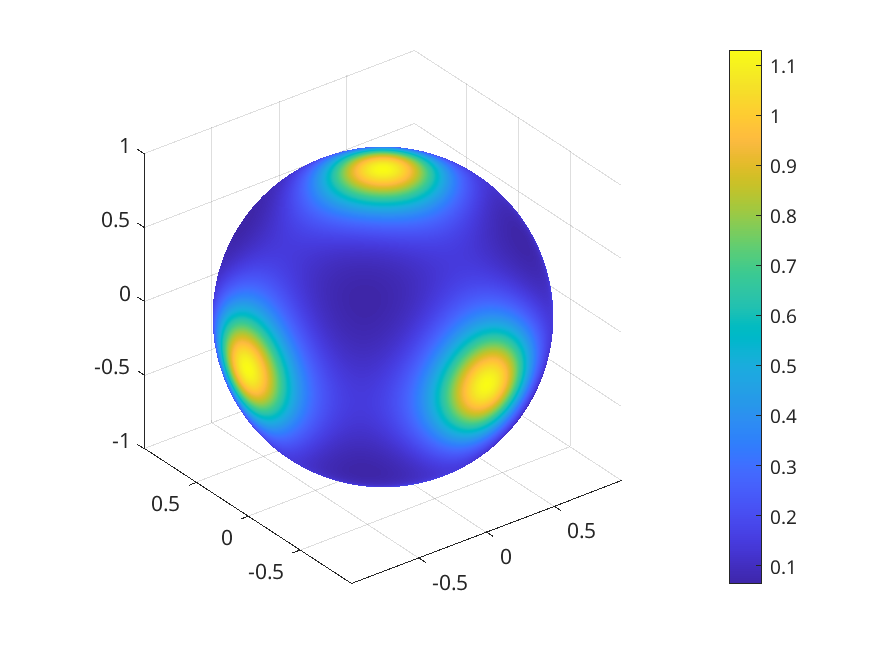}} \qquad
\subfloat[$O_{5,5}$, B-S entropy, {$[-10,0]$}]{\includegraphics[width=.45\textwidth]{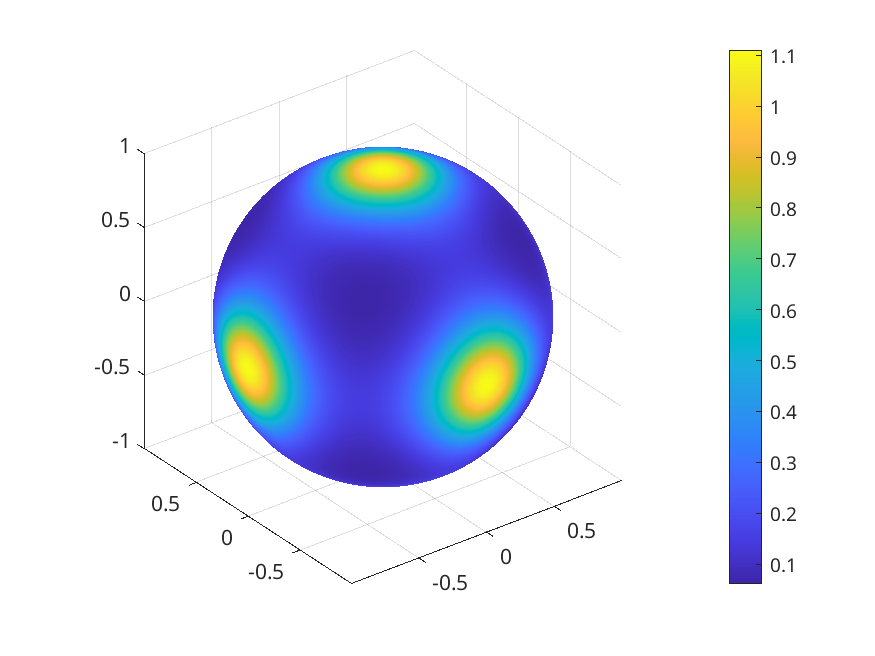}}
\caption{Approximations of the six-Gaussian function based on the Boltzmann-Shannon entropy.}
\label{fig:six_Gaussian_BS}
\end{figure}

Fig.~\ref{fig:six_Gaussian_BS_error} shows the error decay as $N$ increases. Although all methods provide spectral convergence, the choice of parameters does affect the convergence rate. In this example, the $O_{N,5}$ model optimized on $[-10,0]$ and the $T_{N,5}$ model with $x_0 = -5$ have similar performance, which is worse than the $O_{N,5}$ model optimized on the interval $[-5,5]$ but better the $T_{N,5}$ model with $x_0 = 0$. Note that the convergence rate of these models is not determined by the quality of approximation to the $\beta$ function. This is also observed in \cite{PhiRTE1}, where the $P_N$ (which corresponds to $\beta_{N,1}$) model has the best convergence rate among all $\beta_{N,K}$ models, although the $\beta_{N,1}$ model only approximates the exponential function by $\mathrm{e}^x \approx 1 + x$, which is a poor fit. Here we conjecture that the $O_{N,5}$ model optimized on $[-5,5]$ has a better convergence rate because the function $O_5(x)$ is relatively closer to a linear function with slope $1$ for a certain range on the negative part of the domain (see Fig.~\ref{fig:TO_BS}). We focus on the negative part because the value of $I(\Omega)$ is mostly between $0$ and $1$. This means the $O_{N,5}$ model optimized on $[-5,5]$ is likely to be closer to the $P_N$ model when approximating this function.

\begin{figure}[!ht]
\centering
\subfloat[Error decay]{\includegraphics[height=.38\textwidth]{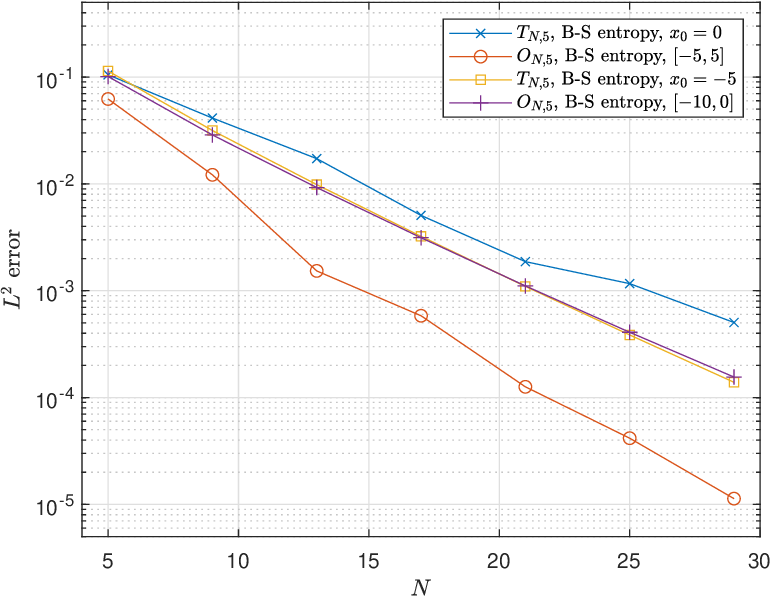}} \quad
\subfloat[Approximation of the exponential]{\label{fig:TO_BS}\includegraphics[height=.38\textwidth]{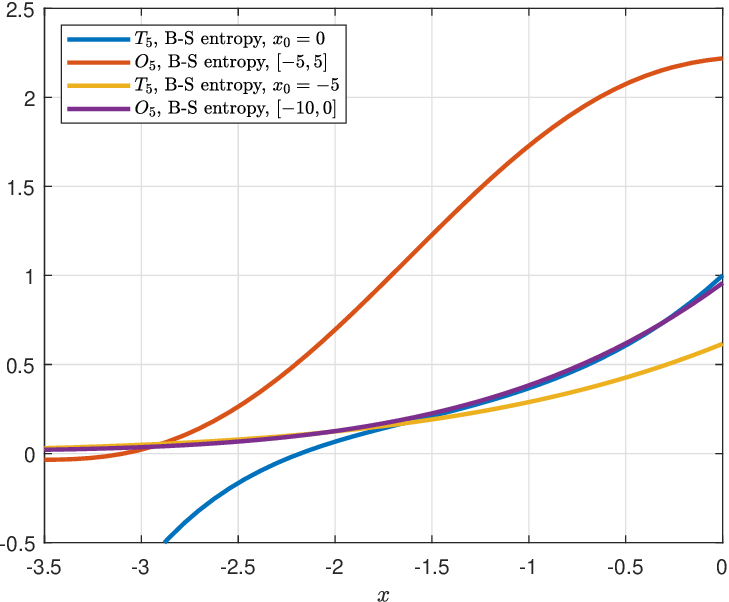}}
\caption{Error decay for approximations of the six-Gaussian function based on the Boltzmann-Shannon entropy and the corresponding approximations of the exponential function}
\label{fig:six_Gaussian_BS_error}
\end{figure}

Now we test the performance of the methods based on the Bose-Einstein entropy. The results are shown in Fig.~\ref{fig:six_Gaussian_BE_error}. It can be seen that the spectral convergence is observed for three models except the $O_{N,5}$ method optimized on the interval $[-5,-0.2]$. This is likely due to the flatness of the $O_5$ function on the interval from $[-2.5, -1.5]$ caused by enforcing good approximations in the region closer to zero where the Planckian has larger values. This can be improved by slightly shifting the upper bound of the domain. Fig.~ \ref{fig:error_decay} gives the result for $O_{N,5}$ method optimized on $[-5, -0.5]$, where a much better convergence rate is obtained. Nevertheless, since the choice of the domain is not optimized, the numerical error is still significantly larger than other lines in Figure \ref{fig:error_decay}.
\begin{figure}[!ht]
\centering
\subfloat[Error decay]{\label{fig:error_decay}\includegraphics[height=.38\textwidth]{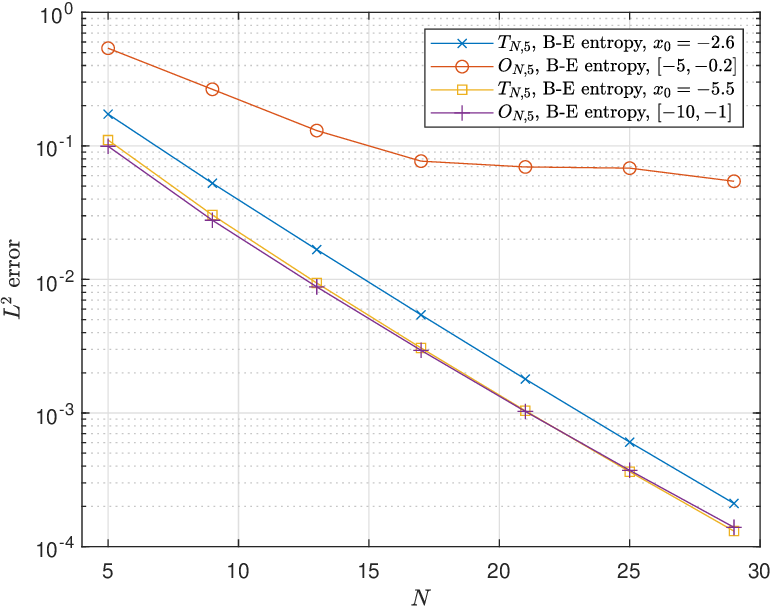}} \quad
\subfloat[Approximation of the exponential]{\label{fig:TO_BE}\includegraphics[height=.38\textwidth]{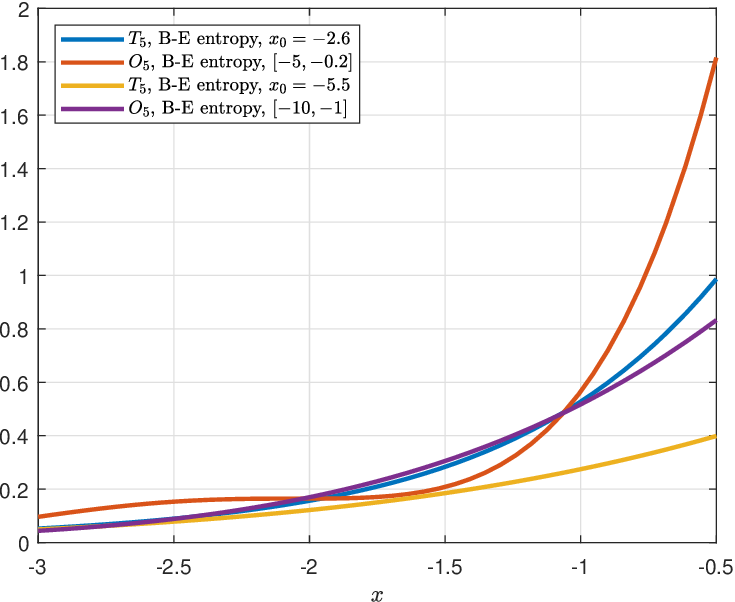}}
\caption{Error decay for approximations of the six-Gaussian function based on the Bose-Einstein entropy and the corresponding approximations of the exponential function}
\label{fig:six_Gaussian_BE_error}
\end{figure}

\begin{figure}[!ht]
\centering
\includegraphics[height=.38\textwidth]{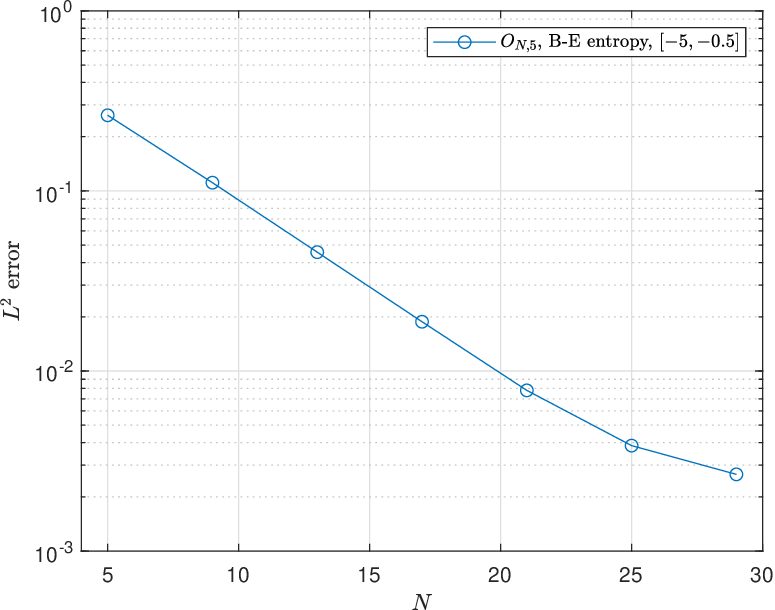}
\caption{Error decay for approximations of the six-Gaussian function based on the Bose-Einstein entropy}
\label{fig:six_Gaussian_BE_5_error}
\end{figure}

\section{Conclusion}
\label{sec:concl}
We have shown that the moment approximation developed in~\cite{PhiRTE1} is flexible enough to be adapted to dynamical models dissipating various types of entropies. Especially, this construction is shown to be suitable to construct a closure dissipating an approximation of a chosen physical entropy such as Boltzmann-Shannon's or Bose-Einstein's entropy. In this work, the closure is constructed using polynomials so that the integration can be carried out exactly in the moment inversion problem. Such an extension introduces many parameters to the approximate models. It is demonstrated by numerical tests that the quality of the moment closure depends on the choice of these parameters, but in most cases, decent results can be obtained by optimizing the distance between the polynomial and the physical entropy. In our future work, we are going to further study their performance by applying these moment models to the radiative transfer equation with interaction with matter.

\section*{Acknowledgements}
Zhenning Cai was supported by the Academic Research Fund of the Ministry of Education of Singapore under grant No. A-0004592-00-00.

\bibliographystyle{plain}
\bibliography{PhiRTE2}

\begin{thebibliography}{10}

\bibitem{abdelmalik_thesis}
M.~Abdelmalik.
\newblock {\em Adaptive algorithms for optimal multiscale model hierarchies of
  the {Boltzmann} Equation: {Galerkin} methods for kinetic theory}.
\newblock PhD thesis, TU Eindhoven, 2017.

\bibitem{PhiRTE1}
M.~Abdelmalik, Z.~Cai, and T.~Pichard.
\newblock Moment methods for the radiative transfer equation based on
  $\varphi$-divergences.
\newblock {\em https://arxiv.org/abs/2304.01758 under review}, pages 1--29,
  2023.

\bibitem{abdelmalik}
M.~Abdelmalik and H.~van Brummelen.
\newblock Moment closure approximations of the boltzmann equation based on
  $\varphi$-divergences.
\newblock {\em J. Stat. Phys.}, 164:77--104, 2016.

\bibitem{spectral_meth2}
C.~Canuto, M.~Y. Hussaini, A.~Quarteroni, and T.~A. Zang.
\newblock {\em Spectral Methods: Fundamentals in Single Domains}.
\newblock Springer-Verlag, 2006.

\bibitem{Carlson_book}
B.~Carlson.
\newblock {\em Methods in computational physics}.
\newblock ACADEMIC PRESS, 1963.

\bibitem{Chandrasekhar_book}
S.~Chandrasekhar.
\newblock {\em Radiative transfer}.
\newblock Oxford, 1950.

\bibitem{czizar}
I.~Csiszár.
\newblock A class of measures of informativity of observation channels.
\newblock {\em Period. Math. Hung.}, 2:191--213, 1972.

\bibitem{Dautray-Lions}
R.~Dautray and J.-L. Lions.
\newblock {\em Mathematical Analysis and Numerical Methods for Science and
  Technology: Volume 6, Evolution Problems {II}}.
\newblock Springer, 2000.

\bibitem{Frank_SPN}
M.~Frank, A.~Klar, E.~Larsen, and S.~Yasuda.
\newblock Time-dependent simplified pn approximation to the equations of
  radiative transfer.
\newblock {\em J. Comput. Phys.}, 226(2):2289--2305, 2007.

\bibitem{Godlewski_Raviart_book}
E.~Godlewski and P.-A. Raviart.
\newblock {\em Numerical approximation of hyperbolic systems of conservation
  laws}.
\newblock Springer, 1996.

\bibitem{Hauck}
C.~D. Hauck.
\newblock High-order entropy-based closures for linear transport in slab
  geometry.
\newblock {\em Commun. Math. Sci}, 9(1):187--205, 2011.

\bibitem{spectral_meth1}
J.~S. Hesthaven, S.~Gottlieb, and D.~Gottlieb.
\newblock {\em Spectral Methods for Time-Dependent Problems}.
\newblock Cambridge, 2009.

\bibitem{Humbird_diffusion}
K.~D. Humbird and R.~McClarren.
\newblock Adjoint-based sensitivity analysis for high-energy density radiative
  transfer using flux-limited diffusion.
\newblock {\em High Energy Density Physics}, 22:12--16, 2017.

\bibitem{Kawashima-Yong}
S.~Kawashima and W.-A. Yong.
\newblock Dissipative structure and entropy for hyperbolic systems of balance
  laws.
\newblock {\em Arch. Rational Mech. Anal.}, 174:345--364, 2004.

\bibitem{Lasserre_book}
J.-B. Lasserre.
\newblock {\em Moment, positive polynomials, and their applications}.
\newblock Imperial college press, 2009.

\bibitem{Lebedev1976quadratures}
V.I. Lebedev.
\newblock Quadratures on a sphere.
\newblock {\em USSR Comput. Math. Math. Phys.}, 16(2):10--24, 1976.

\bibitem{Lebedev1999quadrature}
V.I. Lebedev and D.N. Laikov.
\newblock A quadrature formula for the sphere of the 131st algebraic order of
  accuracy.
\newblock {\em Doklady Mathematics}, 59(3):477--481, 1999.

\bibitem{levermore}
C.~D. Levermore.
\newblock Moment closure hierarchies for kinetic theories.
\newblock {\em J. Stat. Phys.}, 83(5--6):1021--1065, 1996.

\bibitem{Lewis_Miller_book}
E.~E. Lewis and W.~F. Miller.
\newblock {\em Computational Methods of Neutron Transport}.
\newblock John Wiley \& Sons Inc, 1984.

\bibitem{Li_B2}
R.~Li and W.~Li.
\newblock {3D $B_2$} model for radiative transfer equation.
\newblock {\em Int. J. Numer. Anal. Modeling}, 17(1):118--150, 2020.

\bibitem{lowrie-morel}
R.~B. Lowrie, J.~E. Morel, and J.~A. Hittinger.
\newblock The coupling of radiation and hydrodynamics.
\newblock {\em The Astro. J.}, 521:432--450, 1999.

\bibitem{McClaren_SPN}
R.~McClarren.
\newblock Theoretical aspects of the simplified pn equations.
\newblock {\em Transport Theory Stat. Phys.}, 39(2-4):73--109, 2010.

\bibitem{mihalas_book}
D.~Mihalas and B.~R.~W. Mihalas.
\newblock {\em Foundations of Radiation Hydrodynamics}.
\newblock Oxford, 1983.

\bibitem{Minerbo}
G.~N. Minerbo.
\newblock Maximum entropy {Eddington} factors.
\newblock {\em J. Quant. Spectros. Radiat. Transfer}, 20:541--545, 1978.

\bibitem{Olson_diffusion}
G.~L. Olson, L.~H. Auer, and M.~L. Hall.
\newblock Diffusion, p1, and other approximate forms of radiation transport.
\newblock {\em J. Quant. Spec. Rad. Trans.}, 64(6):619--634, 2000.

\bibitem{PiN}
T.~Pichard.
\newblock A moment closure based on a projection on the boundary of the
  realizability domain: 1d case.
\newblock {\em Kin. rel. models}, 13:1243--1280, 2020.

\bibitem{pichard_m2}
T.~Pichard, G.~W. Alldredge, S.~Brull, B.~Dubroca, and M.~Frank.
\newblock An approximation of the {$M_2$} closure: application to radiotherapy
  dose simulation.
\newblock {\em J. Sci. Comput.}, 71:71--108, 2017.

\bibitem{pomraning_book}
G.~C. Pomraning.
\newblock {\em Equations of Radiation Hydrodynamics}.
\newblock Pergamon, 1973.

\bibitem{groth_m2}
J.~A.~R. Sarr and C.~P.~T. Groth.
\newblock A second-order maximum-entropy inspired interpolative closure for
  radiative heat transfer in gray participating media.
\newblock {\em J. Quant. Spectros. Radiat. Transfer}, page 107238, 2020.

\bibitem{Schmuedgen_book}
K.~Schmuedgen.
\newblock {\em The moment problem}.
\newblock Springer, 2017.

\bibitem{Schneider_KN}
F.~Schneider.
\newblock Kershaw closures for linear transport equations in slab geometry i:
  model derivation.
\newblock {\em J. Comput. Phys.}, 322:905--919, 2016.

\bibitem{Szego_book}
G.~Szeg\"o.
\newblock {\em Orthogonal polynomials}.
\newblock AMS, 1939.

\end{thebibliography}

\end{document}